\documentclass{article}

     \usepackage[final,nonatbib]{neurips_2019}

\usepackage[utf8]{inputenc} 
\usepackage[T1]{fontenc}    
\usepackage{hyperref}       
\usepackage{url}            
\usepackage{booktabs}       
\usepackage{amsfonts}       
\usepackage{nicefrac}       
\usepackage{microtype}      

\usepackage{epsfig,epsf,fancybox}
\usepackage{amsmath}
\usepackage{mathrsfs}
\usepackage{amssymb}
\usepackage{graphicx}
\usepackage{color}
\usepackage{multirow}
\usepackage{paralist}
\usepackage{verbatim}
\usepackage{galois}
\usepackage{boxedminipage}
\usepackage{booktabs}
\usepackage{accents}
\usepackage{stmaryrd}
\usepackage{float}
\usepackage{lipsum}
\usepackage{subcaption}

\usepackage[linesnumbered,ruled,vlined]{algorithm2e}
\usepackage[noend]{algpseudocode}

\usepackage{array}
\usepackage{booktabs}
\usepackage[dvipsnames]{xcolor}
\usepackage{tikz}
\usepackage{wrapfig}

\newcolumntype{M}[1]{>{\centering\arraybackslash}m{#1}}

\newtheorem{theorem}{Theorem}[section]

\newtheorem{lemma}[theorem]{Lemma}

\newcommand{\R}{\mathbb{R}}

\title{Interior-point Methods Strike Back: Solving the Wasserstein Barycenter Problem}

\author{
  Dongdong Ge\\
  Research Institute for Interdisciplinary Sciences\\
  Shanghai University of Finance and Economics\\
  \texttt{ge.dongdong@mail.shufe.edu.cn}\\
  \And
  Haoyue Wang\thanks{{Haoyue Wang and Zikai Xiong are corresponding authors.}}
  \\
  School of Mathematical Sciences\\
  Fudan University\\
  \texttt{haoyuewang14@fudan.edu.cn}  \\
  \And
  Zikai Xiong\footnotemark[1]
  \\
  School of Mathematical Sciences\\
  Fudan University\\
  \texttt{zkxiong16@fudan.edu.cn}  \\
  \And
  Yinyu Ye\\
  Department of Management Science and Engineering\\
  Stanford University\\
  \texttt{yyye@stanford.edu}
}

\begin{document}

\maketitle

\begin{abstract}
Computing the Wasserstein barycenter of a set of probability measures under the optimal transport metric can quickly become prohibitive for traditional second-order algorithms, such as interior-point methods, as the support size of the measures increases. In this paper, we overcome the difficulty by developing a new adapted interior-point method that fully exploits the problem's special matrix structure to reduce the iteration complexity and speed up the Newton procedure. Different from regularization approaches, our method achieves a well-balanced tradeoff between accuracy and speed. A numerical comparison on various distributions with existing algorithms exhibits the computational advantages of our approach. Moreover, we demonstrate the practicality of our algorithm on image benchmark problems including MNIST and Fashion-MNIST.

\end{abstract}

\section{Introduction}

To compare, summarize, and combine probability measures defined on a space is a fundamental task in statistics and machine learning.
Given support points of probability measures in a metric space and a transportation cost function (e.g. the Euclidean distance), Wasserstein distance defines a distance between two measures as the minimal transportation cost between them. This notion of distance leads to a host of important applications, including text classification \cite{word embedding},
clustering \cite{Cluster,Ho-Cluster2,soft clustering,persistence diagrams},  unsupervised learning \cite{GAN,domain adaption}, semi-supervised learning \cite{semi-supervised}, supervised-learning\cite{supervised word mover,learning with wa loss}, statistics \cite{Bayes1, Bayes2,Bayes3,hp testing}, and others \cite{Ads,kernel,robust,dic learning,point embedding}.
Given a set of measures in the same space, the 2-Wasserstein barycenter is defined as the measure minimizing the sum of squared 2-Wasserstein distances to all measures in the set. For example, if a set of images (with common structure but varying noise) are modeled as probability measures, then the Wasserstein barycenter is a mixture of the images that share this common structure. The Wasserstein barycenter better captures the underlying geometric structure than the barycenter defined by the Euclidean or other distances. As a result, the Wasserstein barycenter has applications in clustering \cite{Cluster,Ho-Cluster2,soft clustering}, image retrieval \cite{Cuturi-2013-OT} and others \cite{minimax learning,texture,matching,gaussian process}.


From the computation point of view, finding the barycenter of a set of discrete measures can be formulated by linear programming\cite{LP-formulation,LPbook}. Nonetheless, state-of-the-art linear programming solvers do not scale with the immense amount of data involved in barycenter calculations. Current research on computation mainly follows two types of methods. The first type attempts to solve the linear program (or some equivalent problem) with scalable first-order methods. J.Ye et al. \cite{BADMM} use modified Bregman ADMM(BADMM) -- introduced by \cite{nips-BADMM} -- to compute Wasserestein barycenters for clustering problems. L.Yang et al. \cite{sGS-ADMM} adopt symmetric Gauss-Seidel ADMM to solve the dual linear program, which reduces the computational cost in each iteration. S.Claici et al. \cite{Stochastic WB} introduce a stochastic alternating algorithm that can handle continuous input measures. 
However, these methods are still computationally inefficient when the number of  support points of the input measures and the number of input measures are large. Due to the nature of the first-order methods, these algorithms often converge too slowly to reach high-accuracy solutions.

The second, more mainstream, approach introduces an entropy regularization term to the linear programming formulation\cite{Cuturi-2013-OT,IBP}. This technique was first developed in solving optimal transportation problem. See \cite{Cuturi-2013-OT, nips2017-near linear, AGD is better, Optimal Transport book,Lin-2019,optimal complexity,nips2016-stochastic} for the related works. M. Staib et al. \cite{Parallel} discuss the parallel computation issue and introduce a sampling method. P.Dvurechenskii et al. \cite{Decentralize} study decentralized and distributed computation for the regularized problem. These methods are indeed suitable for large-scale problems due to their low computational cost and parsimonious memory usage. However, this advantage is obtained at the expense of the solution accuracy: especially when the regularization term is weighted less in order to approximate the original problem more accurately, computational efficiency degenerates and the outputs become unstable \cite{IBP}. S.  Amari et al. \cite{sharpen technique} propose a entropic regularization based sharpening technique but their result is not the accurate real barycenter. P.C. Alvarez-Esteban et al. \cite{fixed point method} prove that the barycenter must be the fixed-point of a new operator. See \cite{computation OT book} for a detailed survey of related algorithms.

In this paper, we develop a new interior-point method (IPM), namely Matrix-based Adaptive Alternating interior-point Method (MAAIPM), to efficiently calculate the Wasserstein barycenters. If the support is pre-specified, we apply one step of the Mizuno-Todd-Ye predictor-corrector IPM\cite{Mizuno-Todd-Ye}. The algorithm gains a quadratic convergence rate showed by Y. Ye et al. \cite{Ye1993}, which is a distinct advantage of IPMs over first-order methods. In practice, we implement Mehrotra's predictor-corrector IPM \cite{pdipm}, and add clever heuristics in choosing step lengths and centering parameters.  If the support is also to be optimized, MAAIPM alternatively updates support and linear program variables in an adaptive strategy. At the beginning, MAAIPM updates support points $X^*$ by an unconstrained quadratic program after a few number of IPM iterations. At the end, MAAIPM updates $X^*$ after every IPM iteration and applies the "jump" tricks to escape local minima. Under the framework of MAAIPM, we present two block matrix-based accelerated algorithms to quickly solve the Newton equations at each iteration.
Despite a prevailing belief that IPMs are inefficient for large-scale cases, we show that such an inefficiency can be overcome through careful manipulation of the block-data structure of the normal equation. As a result, our stylized IPM has the following advantages.

\begin{wrapfigure}[21]{r}{0.25\linewidth}
\vspace{-5pt}
\includegraphics[width=\linewidth]{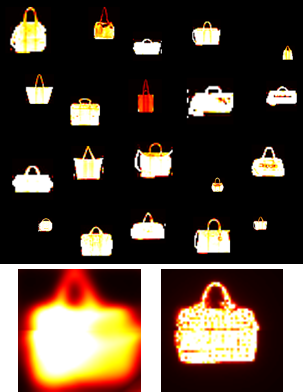}
\caption{A comparison of algorithms for computing the barycenters between a Sinkhorn based approach\cite{IBP}(left) and MAAIPM(right). Samples of handbag(first 4 rows) are from Fashion-MNIST dataset.}
\label{fig:comparison}
\end{wrapfigure}

\textbf{Low theoretical complexity.}
The linear programming formulation of the Wasserstein barycenter has $m\sum_{i=1}^N m_i +m$ variables and{red}$Nm + \sum_{i=1}^N m_i +1$ constraints, where the integers $N$, $m$ and $m_i$ will be specified later. Although MAAIPM is still a second-order method, in our two block matrix-based accelerated algorithms, every iteration of solving the Newton direction has a time complexity of merely $O(m^2 \sum_{i=1}^N m_i + Nm^3)$ or $O( m \sum_{i=1}^N m_i^2 + \sum_{i=1}^N m_i^3)$, where a standard IPM would need
$O\big((Nm + \sum_{i=1}^N m_i +1)^2(m\sum_{i=1}^N m_i +m)\big)$. For simplicity, let $m_i = m, ~i= 1,2\dots,N$, then the time complexity of our algorithm in each iteration is $O(Nm^3)$, instead of standard IPM's complexity $O(N^3m^4)$.
Note that theoretically, when $Nm^2 = (Nm)^k$ for some $1<k<2$, the complexity of the standard IPM can be reduced to $O((Nm)^{\omega(k)}) + O((Nm)^3)$ via fast matrix computation methods, where the specific value of $\omega(k)$ can be found in table 3 of \cite{Fast-matrix-multiplication})

\textbf{Practical effectiveness in speed and accuracy.}
    Compared to regularized methods, IPMs gain high-accuracy solutions and high convergence rate by nature. { Numerical experiments show that our algorithm converges to highly accurate solutions of the original linear program with the least number of iterations. } Figure \ref{fig:comparison} shows the advantages of our methods in accuracy in comparison to the well-developed Sinkhorn-type algorithm \cite{Cuturi-2013-OT,IBP}.




There are more advantages of our approaches in real implementation. When the support points of measures are the same, there are several specially designed highly memory-efficient {and thus very fast} Sinkhorn based algorithms such as \cite{convolutional sinkhorn,IBP}. However, when the support points of measures are different, the convolutional method in \cite{convolutional sinkhorn} is no longer applicable and the memory usage of our method is within a constant multiple of { the popular memory-efficient first-order Sinkhorn method}, IBP\cite{IBP},  much less than the memory used by a commercial solver. {In this case, experiments also show that our algorithm can perform the best in both accuracy and overall runtime.} 
Our algorithms also inherits a natural structure potentially fitting parallel computing scheme well. Those merits ensure that our algorithm is highly suitable for large-scale computation of Wasserstein barycenters.

The rest of the paper is organized as follows. In section 2, we briefly define the Wasserstein barycenter. In section 3, we present its linear programming formulation and introduce the IPM framework. In section 4, we present an IPM implementation that greatly reduces the computational cost of classical IPMs. In section 5, we present our numerical results.

\section{Background and Preliminaries}

In this section, we briefly recall the Wasserstein distance and the Wasserstein barycenter for a set of discrete probability measures \cite{Agueh-2011,Cuturi-2014-OTbarycenter}.
Let $\Sigma_n = \{ \boldsymbol{a} \in \R^n | \sum_{i=1}^{n}a_i=1, a_i \ge 0 \text{ for }  i =1, 2, \dots  n\}$ be the probability simplex in $\R^n$.
For two vectors $\boldsymbol{s}^{(1)} \in \Sigma_{n_1} ,\boldsymbol{s}^{(2)} \in \Sigma_{n_2} $, define the set of matrices $ \mathcal{M}( \boldsymbol{s}^{(1)}, \boldsymbol{s}^{(2)} ) = \{ \Pi \in \R_+^{n_1 \times n_2} :  \Pi\boldsymbol{1}_{n_2} = \boldsymbol{s}^{(1)}  , \Pi^{\top}\boldsymbol{1}_{n_1} = \boldsymbol{s}^{(2)} \} $. Let $\mathcal{P}=\{ (a_i,\boldsymbol{q}_i): i=1,\dots,m \}$ denote the discrete probability measure supported on $m$ points $\boldsymbol{q}_1,\dots , \boldsymbol{q}_m$ in $\mathbb{R}^d$ with weights $a_1,\dots,a_m $ respectively. The Wasserstein barycenter of the two measures $\mathcal{U}=\{(a_i,\boldsymbol{q}_i): i=1,\dots, m_1\}$ and $\mathcal{V}=\{(b_j,\boldsymbol{p}_j): j=1,\dots, m_2\}$ is
\begin{equation}\label{def of W-dist}
  \mathcal{W}_2( \mathcal{U}, \mathcal{V} ) : = \min \left\{ \sqrt{  \sum_{i=1}^{m_1} \sum_{j=1}^{m_2} \pi_{ij}\| \boldsymbol{q}_i - \boldsymbol{p}_j\|^2  } : \Pi = [\pi_{ij}] \in \mathcal{M}( \boldsymbol{a},\boldsymbol{b} ) \right\}
\end{equation}
where $ \boldsymbol{a} = (a_1,\dots,a_{m_1})^{\top} $ and $ \boldsymbol{b} = (b_1,\dots,b_{m_2})^{\top} $.
Consider a set of probability measures $\{ \mathcal{P}^{(t)}, t=1,\cdots, N \}$ where $\mathcal{P}^{(t)}= \{ (a_i^{(t)}, \boldsymbol{q}_i^{(t)}) : i=1,\dots, m_t \} $, and let $\boldsymbol{a}^{(t)} = (a_1^{(t)}, \dots,a_{m_t}^{(t)} )^{\top} $. The Wasserstein barycenter (with $m$ support points) $\mathcal{P}= \{ (w_i, \boldsymbol{x}_i): i=1,\cdots, m \}$ is another probability measure which is defined as a solution of the problem
\begin{equation}\label{def of W-bartcenter}
  \min\limits_{\mathcal{P}} \frac{1}{N} \sum_{t=1}^N ( \mathcal{W}_2( \mathcal{P},\mathcal{P}^{(t)} ) )^2.
\end{equation}
Furthermore, define the simplex $ \mathcal{S}= \big\{ (\boldsymbol{w}, \Pi^{(1)}, \dots,  \Pi^{(N)})\in \R^m_+ \times \R^{m\times m_1}_+ \times \cdots \times \R^{m\times m_N}_+:
{ \boldsymbol{1} _ { m } ^ { \top } \boldsymbol { w } = 1 , \boldsymbol { w } \geq 0 }; \
{ \Pi ^ { ( t ) } \boldsymbol{1} _ { m _ { t } } = \boldsymbol { w } , \left( \Pi ^ { ( t ) } \right) ^ { \top } \boldsymbol{1} _ { m } =  \boldsymbol{a} ^ { ( t ) } , \Pi ^ { ( t ) } \geq 0 , \forall t = 1 , \cdots , N }
\big\}. $
For a given set of support points $X = \{\boldsymbol{x}_1, \dots, \boldsymbol{x}_m\}$, define the distance matrices $D^{(t)}(X)= [ \|\boldsymbol{x}_i - \boldsymbol{q}^{(t)}_j\|^2_2 ] \in \R^{m\times m_t}$ for $t=1,\dots,N$.
Then problem \eqref{def of W-bartcenter} is equivalent to
\begin{equation}\label{non-fix}
    \min\limits_{\boldsymbol{w},X,\Pi^{(t)}} \ \
    \sum _ { t = 1 } ^ { N } \left\langle D^{(t)}(X) , \Pi ^ { ( t ) } \right\rangle \ \
    \text{s.t.} \ \ (\boldsymbol{w},\Pi^{(1)},\dots,\Pi^{(N)}) \in \mathcal{S}, \ \boldsymbol{x}_1, \dots, \boldsymbol{x}_m \in \R^n.
\end{equation}
Problem (\ref{non-fix}) is a nonconvex problem, where one needs to find the optimal support points X and the optimal weight vector $\boldsymbol{w}$ of a barycenter simultaneously. However, in many real applications, the support X of a barycenter can be specified empirically from the support points of $\{ \mathcal{P}^{(t)} \}_{t=1}^N$. Indeed, in some cases, all measures in $\{ \mathcal{P}^{(t)} \}_{t=1}^N$ have the same set of support points and hence the barycenter should also take the same set of support points. In view of this, we will also focus on the case when the support X is given. Consequently, problem (\ref{non-fix}) reduces to the following problem:
\begin{equation}\label{fix}
    \min\limits_{\boldsymbol{w},\Pi^{(t)}} \ \
    \sum _ { t = 1 } ^ { N } \left\langle D^{(t)} , \Pi ^ { ( t ) } \right\rangle \ \
    \text{s.t.} \ \ (\boldsymbol{w},\Pi^{(1)},\dots,\Pi^{(N)}) \in \mathcal{S}
\end{equation}
where $D^{(t)}$ denotes $ \mathcal{D}(X,Q^{(t)})$ for simplicity. In the following sections, we refer to problem (\ref{fix}) as the \textit{Pre-specified Support Problem}, and call problem (\ref{non-fix}) the \textit{Free Support Problem}.

\section{General Framework for MAAIPM}\label{fixed support problem}
\paragraph{Linear programming formulation and preconditioning.}
Note that the Pre-specified Support Problem is a linear program. In this subsection, we focus on removing redundant constraints.
First, we vectorize the constraints $ \Pi ^ { ( t ) } \boldsymbol { 1 }_{ m _ { t } } = \boldsymbol { w } $ and $\left( \Pi ^ { ( t ) } \right) ^ { \top } \boldsymbol { 1 } _ { m } = \boldsymbol { a } ^ { ( t ) } $ captured in $\mathcal S$ to become
\begin{equation*}
  ( \boldsymbol { 1 }_{m_t}^{\top} \otimes I_{m} ) vec( \Pi ^ {(t)} ) = \boldsymbol { w } ,\ \  ( I_{m_t}\otimes \boldsymbol { 1 }_{m}^{\top} ) vec( \Pi ^ {(t)} ) = \boldsymbol { a }^{(t)}, \ \ t = 1 , \cdots , N.
\end{equation*}
Thus, problem (\ref{fix}) can be formulated into the standard-form linear program:
\begin{equation}\label{Standard LP}
  \min \ \boldsymbol{c}^{\top}\boldsymbol{x} \ \ \text{s.t.} \ \ A\boldsymbol{x}=\boldsymbol{b}, \boldsymbol{x}\ge 0
\end{equation}
with $ \boldsymbol{x}=(vec(\Pi^{(1)}) ; ... ; vec(\Pi^{(N)}); \boldsymbol{w} ) $ , $b= ( \boldsymbol{a}^{(1)}; \boldsymbol{a}^{(2)}; ... ; \boldsymbol{a}^{(N)}; \boldsymbol{0}_m; ...; \boldsymbol{0}_m; 1 )$,
$\boldsymbol{c} =  (vec(D^{(1)}) ; ... ; vec(D^{(N)}); \boldsymbol{0} ) $
and $ A=
\begin{bmatrix} E_1^{\top} & E_2^{\top} & 0 \\ 0 & E_3^{\top} & \boldsymbol{1}_m \end{bmatrix}^{\top}$,
where $E_1$ is a block diagonal matrix: $E_1=diag(I_{m_1}\otimes \boldsymbol{1}_m^{\top},..., I_{m_N}\otimes \boldsymbol{1}_m^{\top})$; $E_2$ is a block diagonal matrix: $E_2= diag(\boldsymbol{1}_{m_1}^{\top}\otimes I_m, ..., \boldsymbol{1}_{m_N}^{\top}\otimes I_m)$; and $E_3 = -\boldsymbol{1}_N \otimes I_m$.
Let $M:=\sum_{i=1}^N m_i $, $ n_{row} := Nm + \sum_{i=1}^N m_i +1$ and $ n_{col} := m\sum_{i=1}^N m_i +m  $. Then $A \in \R^{n_{row} \times n_{col}}, \boldsymbol{b} \in \R^{n_{row}}$ and $ \boldsymbol{c} \in \R^{n_{col}}$.
We are faced with a standard form linear program with $n_{col}$ variables and $n_{row}$ constraints. In the spacial case where all $m_i = m$, the number of variables is $O(Nm)$, and the number of constraints is $O(Nm^2)$.

For efficient implementations of IPMs for this linear program, we need to remove redundant constraints.

\begin{lemma}\label{precondition-lemma}
Let $\bar A \in \R^{(n_{row}-N)\times n_{col}}$ be obtained from $A$ by removing the $(M+1)$-th, $(M+m+1)$-th, $\cdots$, $(M+(N-1)m+1)$-th  rows of $A$, and  $\bar{\boldsymbol{b}} \in \R^{n_{row}-N}$ be obtained from $\boldsymbol{b}$ by removing the $(M+1)$-th, $(M+m+1)$-th, $\cdots$, $(M+(N-1)m+1)$-th entries of $\boldsymbol{b}$. Then
 1) $\bar A $ has full row rank; \quad
2) $\boldsymbol{x}$ satisfies $A\boldsymbol{x} = \boldsymbol{b}$ if and only if $\boldsymbol{x}$ satisfies $ \bar A\boldsymbol{x} = \bar{\boldsymbol{b}}$.
\end{lemma}

The proof of this lemma is available in the appendix. With this lemma, the primal problem and dual problem of problem \ref{Standard LP} can be written as
\begin{equation}\label{Regularized Standard LP}
\text{(Primal)}  \min \ \boldsymbol{c}^{\top}\boldsymbol{x} \ \ \text{s.t. } \ \ \bar{A} \boldsymbol{x}=\bar{\boldsymbol{b}}, \boldsymbol{x}\ge 0.
\quad\text{(Dual)}  \max \ \boldsymbol{\bar b}^{\top}\boldsymbol{p} \ \ \text{s.t. } \ \ \bar A^{\top}\boldsymbol{\lambda}+\boldsymbol{s} =\boldsymbol{c}, \boldsymbol{s}\ge 0.
\end{equation}

\paragraph{Framework of Matrix-based Adaptive Alternating Interior-point Method (MAAIPM).}
When the support points are not pre-specified, we need to solve problem \eqref{non-fix}. As we just saw, When $X$ is fixed, the problem becomes a  linear program. When $(\boldsymbol{w},\{\Pi^{(t)}\})$ are fixed, the problem is a quadratic optimization problem with respect to $X$, and the optimal $X^*$ can be written in closed form as
\begin{equation}\label{optimalX}
  \textstyle\boldsymbol{x}_i^* = \big(\sum_{t=1}^{N}\sum_{j=1}^{m_t}\pi_{ij}^{(t)}\big)^{-1}{\sum_{t=1}^{N}\sum_{j=1}^{m_t}\pi_{ij}^{(t)}\boldsymbol{q}_{j}^{(t)}},\quad i=1,2\ldots,m.
\end{equation}
In anther word, \eqref{non-fix} can be reformulated as
\begin{equation}\label{re-nonfix}
  \min \boldsymbol{c}(\boldsymbol{x})^{\top}\boldsymbol{x} \ \ \text{s.t.} \bar{A}\boldsymbol{x} = \bar{\boldsymbol{b}}, ~\boldsymbol{x}\ge 0.
\end{equation}
Since, as stated above, \eqref{non-fix} is a non-convex problem and so it contains saddle points and local minima. This makes finding a global optimizer difficult. Examples of local minima and saddle points are available in the appendix.
The alternating minimization strategy used in \cite{Cuturi-2014-OTbarycenter,sGS-ADMM,BADMM}
alternates between optimizing $X$ by solving \eqref{optimalX} and optimizing $(\boldsymbol{w},\{\Pi^{(t)}\})$ by solving (\ref{fix}). However, this alternating approach cannot avoid local minima or saddle points. Every iteration may require solving a linear program \eqref{fix}, which is expensive when the problem size is large.

To overcome the drawbacks, we propose Matrix-based Adaptive Alternating IPM (MAAIPM).
If the support is pre-specified, we solve a single linear program by predictor-corrector IPM\cite{pdipm, book, PrimalDualIPM}.
If the support should be optimized, MAAIPM uses an adaptive strategy.
At the beginning, because the primal variables are far from the optimal solution, MAAIPM updates $X^*$ of \eqref{optimalX} after a few number of IPM iterations for $(\boldsymbol{w},\{\Pi^{(t)}\})$. Then, MAAIPM updates $X^*$ after every IPM iteration and applies the "jump" tricks to escape local minima. Although MAAIPM cannot ensure finding a globally optimal solution, it can frequently get a better solution in shorter time. Since at the beginning MAAIPM updates $X^*$ after many IPM iterations, primal dual predictor-corrector IPM  is more efficient. At the end, $X^*$ is updated more often
and each update of $X^*$ changes the linear programming objective function so that dual variables may be infeasible. However, the primal variables always remain feasible so that the primal IPM is more suitable at the end. Moreover, primal IPM is better for applying "jump" tricks or other local-minima-escaping techniques, which has been shown in \cite{Ye-nonconvex}. Details and illustration are available in the appendix.

In predictor-corrector IPM, the main computational cost lies in solving the Newton equations, which can be reformulated as the normal equations
\begin{equation}\label{ADAsystem}
    \bar{A}({D^k})^2\bar{A}^{\top}\Delta \boldsymbol{\lambda}^k = \boldsymbol{f}^k,
\end{equation}
where $D^k$ denotes ${diag}(x_i^{(k)}{}/s_i^{(k)})$ and $\boldsymbol{f}^k$ is in $\mathbb{R}^{n_{row}-N}.$ This linear system of matrix $\bar{A}({D^k})^2\bar{A}^{\top}$ can be efficiently solved by the two methods proposed in the next section.
In the primal IPM, MAAIPM combines following the central path with optimizing the support points, i.e., it contains three parts in one iteration, taking an Newton step in the logarithmic barrier function
\begin{equation}\label{barrier function}
\textstyle  \text{minimize}\quad \boldsymbol{c}^{\top}\boldsymbol{x}-\mu\sum_{i=1}^{n}\ln {x}_i,\quad \text{subject to}\quad \bar{A}\boldsymbol{x} = \boldsymbol{b},
\end{equation}
reducing the penalty $\mu$, and updating the support \eqref{optimalX}.
The Newton direction $\boldsymbol{p}_k$ at the $k^{th}$ iteration is calculated by
\begin{equation}\label{Newton direction}
\textstyle  \boldsymbol{p}^k = \boldsymbol{x}^k+(X^k)^2\Big(\bar{A}^{\top}{\big(\bar{A}(X^k)^2\bar{A}^{\top}\big)}^{-1}\big(\bar{A}(X^k)^2\boldsymbol{c}-\mu \bar{A}X^k\boldsymbol{1}\big)-\boldsymbol{c}\Big)/\mu^k,
\end{equation}
where $X^k = {diag}(x_i^{(k)})$. The main cost of primal IPM lies in solving a linear system of $\bar{A}(X^k)^2\bar{A}^{\top}$, which again can be efficiently solved by the two methods described in the following section.
Further more, we also apply the warm-start technique to smartly choose the starting point of the next IPM after "jump" \cite{Yewarmstart}.
Compared with primal-dual IPMs' warm-start strategies \cite{warm start, imple warm}, our technique saves the searching time, and only requires slightly more memory.
When we suitably set the termination criterion, numerical studies show that MAAIPM outperforms previous algorithms in both speed and accuracy, no matter whether  the support is pre-specified or not.

\section{Efficient Methods for Solving the Normal Equations}\label{Efficient Methods for Solving the Normal Equations}

In this section, we discuss efficient methods for solving normal equations in the format $(\bar A D \bar A^{\top})\boldsymbol{z} = \boldsymbol{f}$, where $D$ is a diagonal matrix with all diagonal entries being positive. Let $\boldsymbol{d} =
{diag}(D)$, and $M_2 = N(m-1)$. First, through simple calculation, we have the following lemma on the structure of matrix $\bar A D \bar A^{\top}$, whose proof is available in the appendix.

\begin{lemma}\label{structure1}
$\bar A D \bar A^T$ can be written in the following format:
\begin{equation*}
\bar A D \bar A^T =
\begin{bmatrix}
B_1  &  B_2  &  \boldsymbol{0} \\
B_2^{\top} & B_3 + B_4 & \boldsymbol{\alpha} \\
\boldsymbol{0} & \boldsymbol{\alpha}^{\top} & c
\end{bmatrix}
\end{equation*}
where $B_1\in \R^{M\times M}$ is a diagonal matrix with positive diagonal entries; $B_2 \in \R^{M\times M_2}$ is a block-diagonal matrix with N blocks (the size of the i-th block is $(m-1)\times m_i$); $B_3 \in \R^{M_2\times M_2}$ is a diagonal matrix with positive diagonal entries; Let $\boldsymbol{y} = \boldsymbol{d}(n_{col} - m+2 : n_{col})$, then $B_4 = (\boldsymbol{1}_N \boldsymbol{1}_N^{\top} ) \otimes diag(\boldsymbol{y})$, and $ \boldsymbol{\alpha} = -\boldsymbol{1}_N \otimes \boldsymbol{y}$; $ c = \boldsymbol{1}_{m}^{\top} \boldsymbol{d}(n_{col} - m+1 : n_{col}) $.
\end{lemma}

\paragraph{Single low-rank regularization method (SLRM).}
Briefly speaking, we will perform several basic transformations on the matrix $\bar A D \bar A^T$ to transform it into an easy-to-solve format. Then we solve the system with the transformed coefficient matrix and finally transform the obtained solution back to get an solution of $(\bar A D \bar A^{\top})\boldsymbol{z} = \boldsymbol{f}$.

Define $
V_1:=
  \begin{bmatrix}
  I_M   &   &   \\
  -B_2^{\top}B_1^{-1} & I_{M_2} & \\
     &   &  1
  \end{bmatrix}, \ \
  V_2:=
  \begin{bmatrix}
  I_M   &   &   \\
        & I_{M_2} & -\boldsymbol{\alpha}/c \\
     &   &  1
  \end{bmatrix}
  $,
  $ A_1 := B_3 - B_2^{\top} B_1^{-1} B_2$ and $ A_2 := B_4 - \frac{1}{c} \boldsymbol{\alpha}\boldsymbol{\alpha}^{\top}$.
Then,
$$
V_2V_1 \bar A D \bar A^T V_1^{\top} V_2^{\top} =
\begin{bmatrix}
B_1 & & \\
  &  B_3 - B_2^{\top} B_1^{-1} B_2 + B_4 - \frac{1}{c} \boldsymbol{\alpha}\boldsymbol{\alpha}^{\top} & \\
  &  &  c
\end{bmatrix}
=
\begin{bmatrix}
B_1 & & \\
  &  A_1+A_2 & \\
  &  &  c
\end{bmatrix}.
$$
Define $Y = diag(\boldsymbol{y})- \frac{1}{c}\boldsymbol{y}\boldsymbol{y}^{\top}$,
we have the following lemma.

\begin{lemma}\label{basic2}
\mbox{}\par
\textbf{a)}  $A_1$ is a block-diagonal matrix with N blocks. The size of each block is $(m-1)\times(m-1)$. Further more, $A_1$ is positive definite and strictly diagonal dominant.
\textbf{b)}  $ A_2 = (\boldsymbol{1}_N\boldsymbol{1}_N^{\top})\otimes Y $, and $Y$ is positive definite and strictly diagonal dominant.
\end{lemma}

\begin{wrapfigure}[17]{R}{0.75\textwidth} 
\vspace{-3pt}
\hfill
\begin{minipage}[c]{0.73\textwidth}
\begin{algorithm}[H]
\caption{Solver for the normal equation $(\bar A  D \bar A^T) \boldsymbol{z} =\boldsymbol{f}$}
\label{Normal equation solver}
\KwIn{$\boldsymbol{d} = diag(D) \in \R^{n_{col}}$; $\boldsymbol{f} \in \R^{M+N(m-1)+1} $}

 compute $B_1, B_2, B_3$, vector $\boldsymbol{y} = \boldsymbol{d}(n_{col} - m+2 : n_{col})$ and $c$;

 compute $ T= B_2^{\top}B_1^{-1}$ and matrices $ V_1,V_2 $;

 compute $ A_1 = B_3 - T B_2 $ and $ A_2= (\boldsymbol{1}_N\boldsymbol{1}_N^{\top})\otimes(diag(\boldsymbol{y})- \frac{1}{c}\boldsymbol{y}\boldsymbol{y}^{\top} )$;

 compute $ \boldsymbol{z}^{(1)} = V_1  \boldsymbol{f}$ and $ \boldsymbol{z}^{(2)} =V_2  \boldsymbol{z}^{(1)}$;

 compute $ \boldsymbol{z}^{(3)}(1:M) = B_1^{-1}\boldsymbol{z}^{(2)}(1:M)$;

 compute $\boldsymbol{z}^{(3)}(M+M_2+1) = \frac{1}{c}\boldsymbol{z}^{(2)}(M+M_2+1)$;

 solve the linear system with coefficient matrix $A_1+A_2$ to get $ \boldsymbol{z}^{(3)}(M+1:M+M_2) = (A_1+A_2)^{-1} \boldsymbol{z}^{(2)}(M+1:M+M_2) $;

 compute $ \boldsymbol{z}^{(4)} = V_2^{\top}  \boldsymbol{z}^{(3)}$ , $ \boldsymbol{z} =V_1^{\top}  \boldsymbol{z}^{(4)}$;

 \KwOut{ $\boldsymbol{z}$}

\end{algorithm}
\end{minipage}
\end{wrapfigure}

Since the positive definiteness and diagonal dominance claimed in this lemma, the computation of the inverse matrices of each block of $A_1$ and $A_2$ is numerically stable.
Now we introduce the procedure for solving $ (\bar A  D \bar A^T) \boldsymbol{z} =\boldsymbol{f} $, as descried in Algorithm \ref{Normal equation solver} ($z^{(1)}$ - $z^{(4)}$ in the algorithm are intermediate variables).
In step 7, we need to solve a linear system with coefficient matrix of dimension $N(m-1)\times N(m-1)$, which is hard to compute with common methods for dense symmetric matrices. In view of the low-rank structure of the matrix $A_2$, we introduce a method, namely Single Low-rank Regularization Method (SLRM), which requires only $O(Nm^3)$ flops in computation. Assume $A_1 = diag(A_{11}, A_{22}, ..., A_{NN})$ and define
$ U = \begin{bmatrix} I_{N-1} & \boldsymbol{1}_{N-1} \\ 0 & 1 \end{bmatrix} \otimes I_{m-1}$.
We can solve the linear system $ (A_1+A_2)\boldsymbol{x} = \boldsymbol{g} $ by Algorithm \ref{Low-rank regularization method}.

\begin{wrapfigure}[16]{R}{0.55\textwidth} 
\vspace{-5pt}
\hfill
\begin{minipage}[c]{0.53\textwidth}
\begin{algorithm}[H]
\caption{SLRM for the system $ (A_1+A_2)\boldsymbol{x} = \boldsymbol{g} $}
\label{Low-rank regularization method}
\KwIn{$~A_1,~A_2,~\boldsymbol{g}$}

     compute $A_{ii}^{-1},i=1,..,N$\;

     set $A_1^{-1} = diag( A_{11}^{-1},..., A_{NN}^{-1})$\;

     compute $\boldsymbol{x}^{(1)} = A_1^{-1}\boldsymbol{g}$\;

     compute $\boldsymbol{x}^{(2)} = U^T \boldsymbol{x}^{(1)}$\;

     compute $ \boldsymbol{x}^{(3)}(end-m+2 : end ) =( Y^{-1} + \sum_{i=1}^{N}A_{ii}^{-1} )\backslash \boldsymbol{x}^{(2)}(end-m+2 : end )  $\;

     set $\boldsymbol{x}^{(3)}(1 : end-m+1 )= 0$~\;

      compute $\boldsymbol{x}^{(4)} = U \boldsymbol{x}^{(3)}$ and $\boldsymbol{x}^{(5)} = A_1^{-1} \boldsymbol{x}^{(4)} $\;

      compute $ \boldsymbol{x} = \boldsymbol{x}^{(1)} - \boldsymbol{x}^{(5)} $\;

      \KwOut{$\boldsymbol{x}$}
\end{algorithm}
\end{minipage}
\end{wrapfigure}

The proof of correctness of Algorithm \ref{Low-rank regularization method} and other analysis is available in the appendix.


\paragraph{Double low-rank regularization method (DLRM) when $m$ is large.}\label{further low-rank regularization}
In many applications, $m$ is relatively large compared to $m_t$. For instance, in the area of image identification, the pixel support points of the images at hand are sparse (small $m_t$) but different. To find the "barycenter" of these images, we need to assume the "barycenter" image has much more pixel support points (large $m$) than all the sample images. Sometimes, $m$ might be about 5 to 20 times of each $m_t$. In this case, the computational cost of step 1 in SLRM is heavy, since we need to solve $N$ linear systems with dimension $m\times m$. In this subsection, we use the low rank regularization formula to further reduce the computational cost.

In view of lemma \ref{structure1}, assume
$$
B_1 = diag(B_{11},...,B_{1N}),\ \ B_2 = diag(B_{21},...,B_{2N}), \ \ B_3 = diag(B_{31},...,B_{3N}).
$$
where $B_{1i}\in \R^{m_i \times m_i}$, $ B_{2i} \in \R^{m_i \times (m-1)}$ and $B_{3i} \in \R^{(m-1)\times (m-1)}$. Recall that $A_1 = B_3 - B_2^{\top} B_1^{-1} B_2$ and  $A_1 = diag(A_{11}, ..., A_{NN})$, we have $A_{ii} = B_{3i} - B_{2i}^{\top}B_{1i}^{-1}B_{2i} $. Since $m>>m_i$, we can use the following formula:
\begin{equation}\label{Aii-lowrank-reg}
    A_{ii}^{-1} = (B_{3i} - B_{2i}^{\top}B_{1i}^{-1}B_{2i})^{-1} = B_{3i}^{-1} + B_{3i}^{-1}B_{2i}^{\top} (B_{1i} - B_{2i}B_{3i}^{-1}B_{2i}^{\top})^{-1}B_{2i}B_{3i}^{-1}.
\end{equation}
Instead of calculating and storing each $A_{ii}$ explicitly, we can just calculate and store each $ (B_{1i} - B_{2i}B_{3i}^{-1}B_{2i}^{\top})^{-1} $. When we need to calculate $A_{ii}\boldsymbol{y}$ for some vector $\boldsymbol{y}$, we can use \eqref{Aii-lowrank-reg} and sequentially multiply each matrix with vectors. As a result, the flops required in step 1 of SLRM reduce to $O(m\Sigma_{i=1}^{N}m_i^2 +\Sigma_{i=1}^{N}m_i^3 )$, and the total memory usage of whole MAAIPM is $O(m\Sigma_{i=1}^N m_i)$, which is at the same level (except for a constant) of a primal variable.


\paragraph{Complexity analysis.}
The following theorem summarizes the time and space complexity of the aforementioned two methods.

\begin{theorem}\label{cost}
\textbf{a)} For SLRM,
the time complexity in terms of flops is $O(m^2 \sum_{i=1}^N m_i + Nm^3)$, and the memory usage in terms of doubles is $O(m\sum_{i=1}^N m_i + Nm^2)$;
\textbf{b)} For the DLRM, the time complexity in terms of flops is $O( m \sum_{i=1}^N m_i^2 + \sum_{i=1}^N m_i^3)$, and the memory usage in terms of doubles is $O(m\sum_{i=1}^N m_i + \sum_{i=1}^N m_i^2)$.
\end{theorem}

We can choose between SLRM and DLRM for different cases to achieve lower time and space complexity. Note that as $N,m,m_i$ grows up, the memory usage here is within an constant time of the representative Sinkhorn type algorithms like IBP\cite{IBP}.

\section{Experiments}\label{numerical studies}

We conduct three numerical experiments to investigate the real performance of our methods. The first experiment shows the advantages of SLRM and DLRM over traditional approaches in solving Newton equations with a same structure as barycenter problems. The second experiment fully demonstrates the merits of MAAIPM: high speed/accuracy and more efficient memory usage. In the last experiment with real benchmark data, MAAIPM recovers the images better than any other approach implemented.
In different experiments, we compare our methods with state-of-art commercial solvers(MATLAB, Gurobi, MOSEK), the iterative Bregman projection (IBP) by  \cite{IBP}, Bregman ADMM (BADMM) \cite{nips-BADMM,BADMM}. The result also illustrates MAAIPM's superiority over symmetric Gauss-Seidel ADMM (sGS-ADMM) \cite{sGS-ADMM}.

All experiments are run in Matlab R2018b on a workstation with two processors, Intel(R) Xeon(R) Processor E5-2630@2.40Ghz (8 cores and 16 threads per processor) and 64GB of RAM, equipped with 64-bit Windows 10 OS.
Full experiment details are available in the appendix.

\begin{wrapfigure}[21]{r}{0.45\linewidth}
\vspace{-15pt}
    \centering
    \includegraphics[width=1\linewidth]{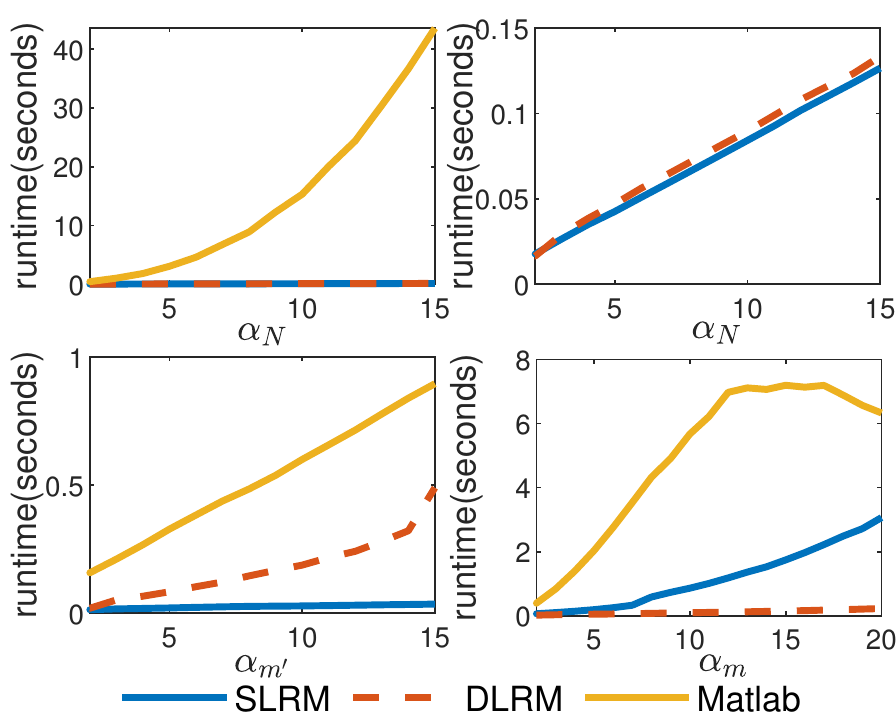}
    \caption{Average computation time of 200 independent trials in solving the linear system. Entries of diagonal $D$ and $\boldsymbol{f}$ are generated by uniform distribution in $(0,1)$. In base situation, $N=50$, $m=50$, $m'=25$. Sub-figures show the computation times when rescaling $N$, $m$ and $m_1=\cdots =m_N=m'$ by respectively $\alpha_N,~\alpha_m$ and $\alpha_{m'}$ times.}
    \label{fig:equations}
\end{wrapfigure}

\vspace{-5pt}

\paragraph{Experiments on solving the normal equations:}
For figure \ref{fig:equations}, one can see that both SLRM and DLRM clealy outperform the Matlab solver in  in all cases. For computation time, SLRM increases linearly with respect to $N$ and $m'$, and DLRM increases linearly with respect to $N$ and $m$, which matches the conclusions in Theorem \ref{cost}.
In practice, we select SLRM when $m^2 \le 4\sum_{t=1}^N m_t^2$ and DLRM when $m^2 > 4\sum_{t=1}^Nm_t^2$.

\vspace{-5pt}

\paragraph{Experiments on barycenter problems:}
In this experiment, we set $d = 3$ for convenience. For $\mathcal{P}^{(t)}$, each entry of $(\boldsymbol{q}_1^{(t)},\ldots,\boldsymbol{q}_{m'}^{(t)})$ is generated with i.i.d. standard Gaussian distribution.
The entries of the weight vectors $(a_1^{(t)},\dots,a_{m'}^{(t)})$ are simulated by uniform distribution on $(0,1)$ and then are normalized.
Next we apply the $k$-means\footnote{We call the Malab function "kmeans" in statistics and machine learning toolbox.} method to choose $m$ points to be the support points.
Note that Gurobi and MOSEK use a crossover strategy  when close to the exact solution to ensure obtaining a highly accurate solution, we can regard Gurobi's objective value $\mathcal{F}_{gu}$ as the exact optimal value of the linear program \eqref{fix}.
Let "normalized obj" denote the normalized objective value defined by
$\vert\mathcal{F}_{method}-\mathcal{F}_{gu}\vert/\mathcal{F}_{gu}$, where $\mathcal{F}_{method}$ is the objective value respectively obtained by each method.
Let "feasibility error" denote $
    \max\Big\{
    \frac{\Vert \{\Pi^{(t)}\boldsymbol{1}_{m_t}-\boldsymbol{w}\}\Vert_F}{1+\Vert \boldsymbol{w}\Vert_F + \Vert\{\Pi^{(t)}\}\Vert_F},
    \frac{\Vert \{(\Pi^{(t)})^{\top}\boldsymbol{1}_{m}-\boldsymbol{a}^{(t)}\}\Vert_F}{1+\Vert \{\boldsymbol{a}^{(t)}\}\Vert_F + \Vert\{\Pi^{(t)}\}\Vert_F},
    \vert \boldsymbol{1}^{\top}\boldsymbol{w}-1\vert
    \Big\}$, as a measure of the distance to the feasible set.

\begin{wrapfigure}[13]{r}{0.5\linewidth}
\vspace{-10pt}
\includegraphics[width=1\linewidth]{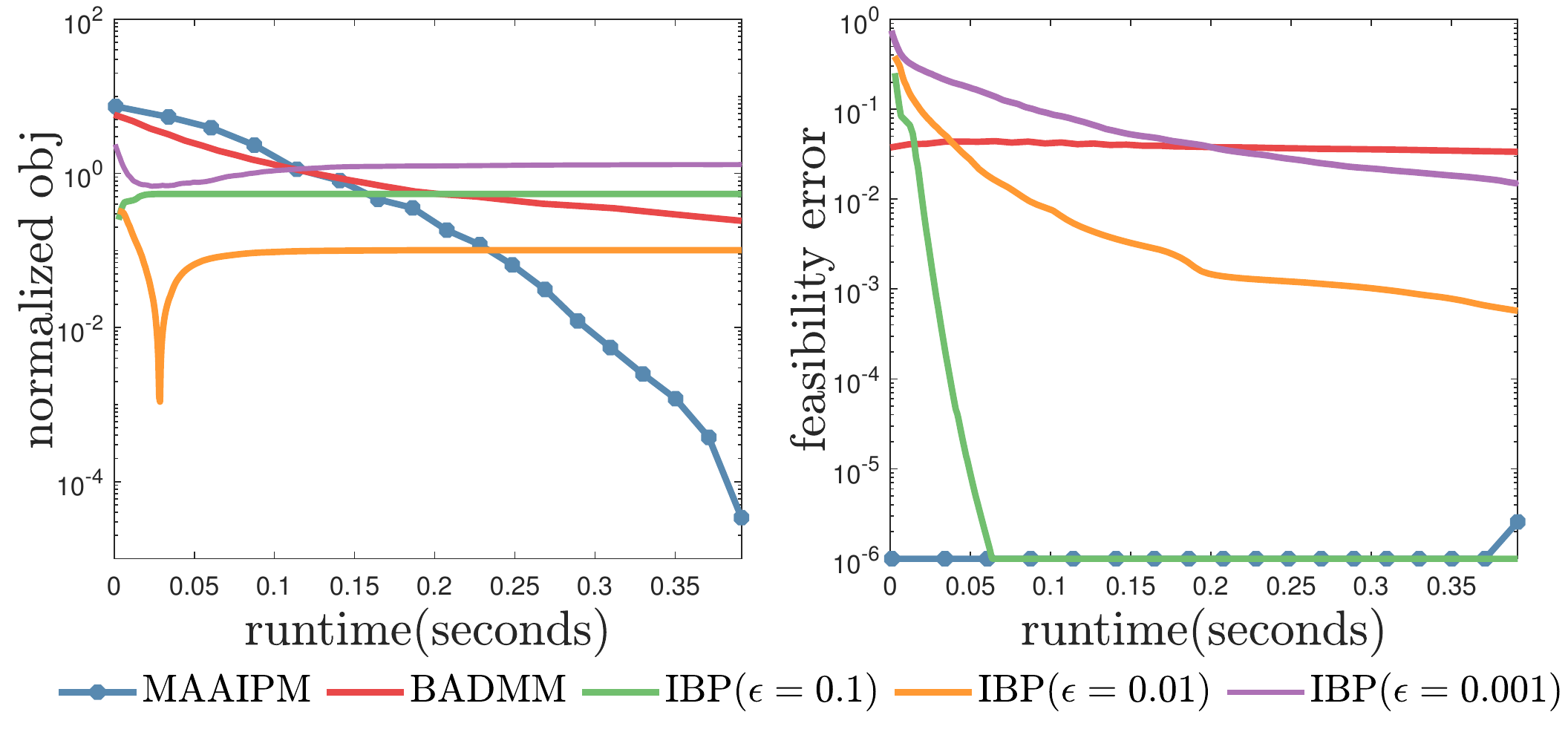}
\caption{Performance of methods in pre-specified support cases. $N = m =50$ and $m_1=\cdots=m_N=50$}
\label{fig:advantage}
\end{wrapfigure}

From figure \ref{fig:advantage}, we see that MAAIPM displays a super-linear convergence rate for the objective, which is consistent with the result of \cite{Ye1993}. Note that the feasibility error of MAAIPM increases a little bit near the end but is still much lower than BADMM and IBP. Although other methods may have lower objective values in early stages, their solutions are not acceptable due to high feasibility errors.

Then we run numerical experiments to test the computation time of methods in pre-specified support points cases.
For MAAIPM, we terminate it when $({\boldsymbol{b}}^{\top}\boldsymbol{\lambda}_k-\boldsymbol{c}^{\top}\boldsymbol{x}_k)/(1+\vert \boldsymbol{b}^{\top}\boldsymbol{\lambda}_k\vert+\vert \boldsymbol{c}^\top \boldsymbol{x}_k \vert)$ is less than $5\times10^{-5}$.
For sGS-ADMM,  we compare with it indirectly by the benchmark claimed in their paper \cite{sGS-ADMM}: commercial solver Gurobi 8.1.0 \cite{Gurobi} (academic license) with the default parameter settings.
We also compare with another commercial solver MOSEK 9.1.0(academic license). In our observation, MAAIPM can frequently perform better than other popular commercial solvers.
We use the default parameter setting(optimal for most cases)  for Gurobi and MOSEK so that they can exploit multiple processors (16 threads) while other methods are implemented with only one thread\footnote{We call the Matlab function "maxNumCompThreads(1)"}.
For BADMM, we follow the algorithm 4 in \cite{BADMM} to implement and terminate when $\Vert \Pi^{(k,1)}-\Pi^{(k,2)}\Vert_{F}/(1+\Vert \Pi^{(k,1)}\Vert_{F} + \Vert \Pi^{(k,2)}\Vert_{F})<10^{-5}$. Set $\Vert \{A_{t}\}\Vert_F = \big(\sum_{t=1}^N\Vert A_t\Vert_F^2\big)^{\frac{1}{2}}$.
For IBP, we follow the remark 3 in  \cite{IBP} to implement the method, terminate it when $\Vert \{u^{(n)}_k\}-\{u^{(n-1)}_k\}\Vert_F/(1+\Vert \{u^{(n)}_k\} \Vert_F + \Vert \{u^{(n-1)}_k\} \Vert_F)<10^{-8}$ and $\Vert \{v^{(n)}_k\}-\{v^{(n-1)}_k\}\Vert_F/(1+\Vert \{v^{(n)}_k\} \Vert_F + \Vert \{v^{(n-1)}_k\} \Vert_F)<10^{-8}$,
and choose the regularization parameter $\epsilon$ from $\{0.1,0.01,0.001\}$ in our experiments.
For BADMM and IBP, we implement the Matlab codes\footnote{Available in  \href{https://github.com/bobye/WBC_Matlab}{https://github.com/bobye/WBC\_Matlab}} by J.Ye et al. \cite{BADMM} and set the maximum iterate number respectively $4000$ and $10^5$.


\vspace{-5pt}

\begin{figure}[htbp]
    \centering    \includegraphics[width=1\linewidth]{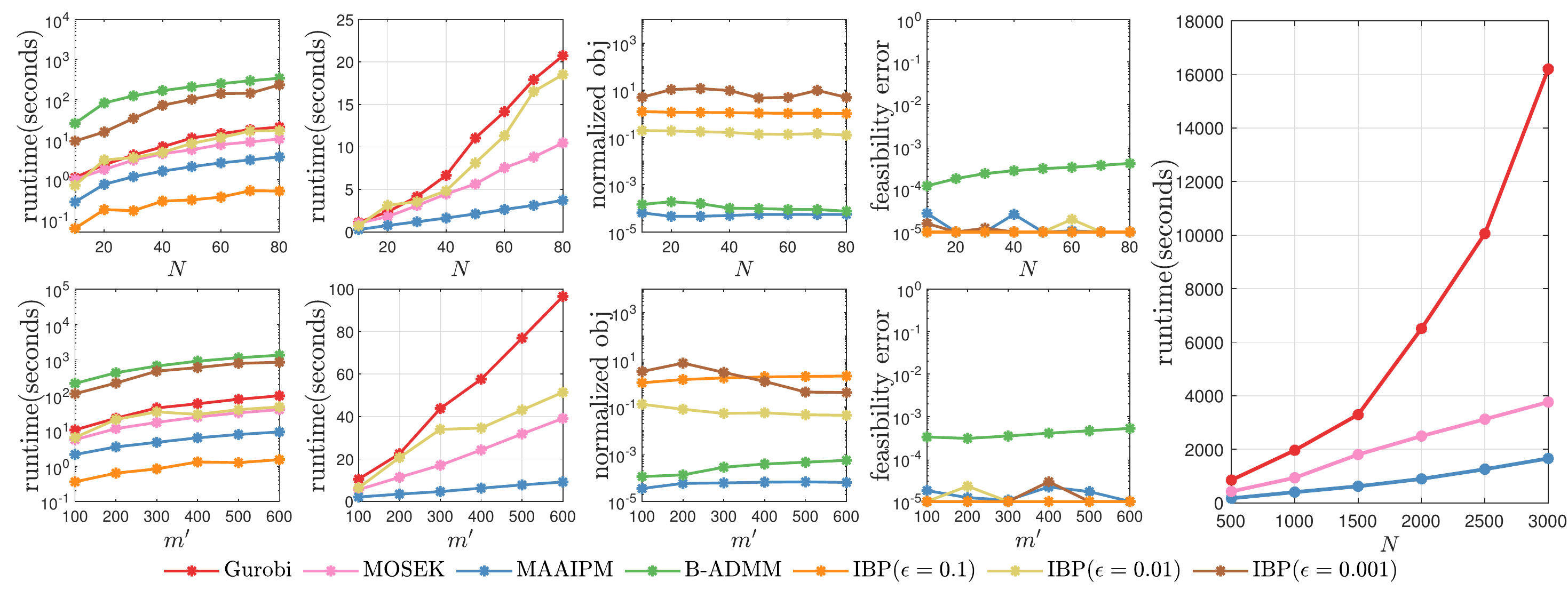}
    \caption{The left 8 figures are the average computation time, normalized objective value and feasibility error of Gurobi, MOSEK, MAAIPM, BADMM and IBP($\epsilon = 0.1,~ 0.01,~ 0.001$) in pre-specified support cases from 30 independent trials. In the first row, $m = 100$, $m_t$ follows an uniform distribution on $(75,125)$. In the second row, $N = 50,~ m =100$ and $m_1=\cdots=m_N=m'$. The right figure is the average computation time of Gurobi and MAAIPM in pre-specified support cases from 10 independent trials. $m_t$ follows a uniform distribution on $(150,250)$, and $m=200$. }
    \label{fig:expNmt}
    \vspace{-5pt}
\end{figure}

\vspace{-5pt}

From the left 8 sub-figures in figure \ref{fig:expNmt} one can observe that MAAIPM returns a considerably accurate solution in the second shortest computation time.
For IBP, although it returns an objective value in the shortest time when $\epsilon=0.1$, the quality of the solution is almost the worst.
Because IBP only solves an approximate problem, if $\epsilon$ is set smaller, the computation time sharply increases but the quality of the solution is still not ensured.
For BADMM, it gives a solution close to the exact one, but requires much more computation time.

\begin{wrapfigure}[18]{r}{0.54\linewidth}
\vspace{-10pt}
    \centering
    \includegraphics[width=1\linewidth]{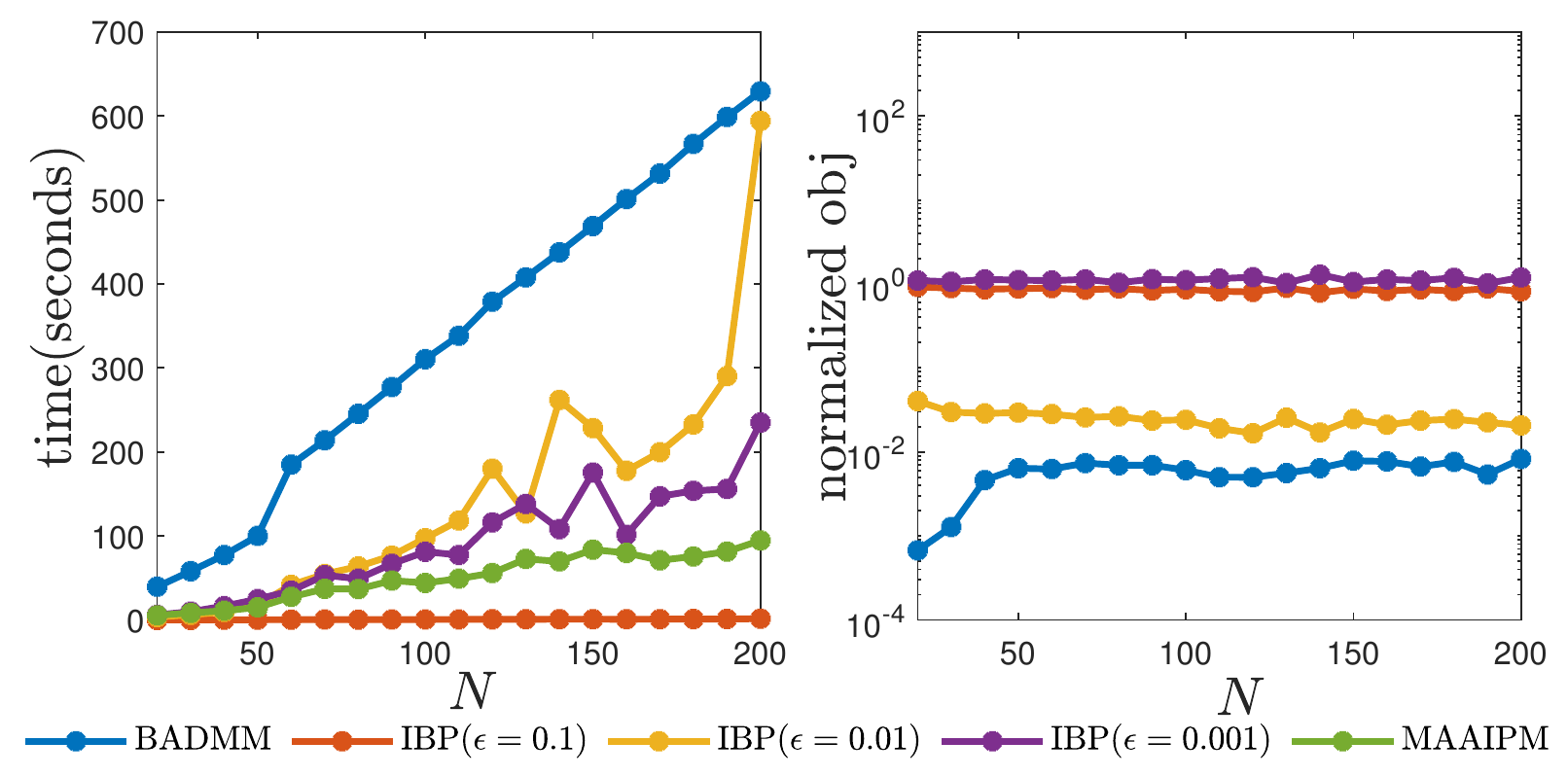}
    \caption{computation time and normalized objective value of MAAIPM, BADMM and IBP in the free support cases from $30$ independent trials. "Normalized obj" denote $\mathcal{F}_{method}/\mathcal{F}_{MAAIPM}-1$, where $\mathcal{F}_{method}$ is the objective value obtained by each method. $N$ takes different values and $m=m'=50$.}
    \label{fig:unfixsupp}
\end{wrapfigure}

For Gurobi and MOSEK, although they can exploit 16 threads, the computation time is far more than that of MAAIPM
That is to say, MAAIPM also largely outperforms sGS-ADMM in speed, according to table 1, 2, 3 in \cite{sGS-ADMM}.
Moreover, because the number of iterations remains almost independent of the problem size, the main computational cost of MAAIPM is approximately linear with respect to $N$ and $m'$. In fact, when $N=5000$, MAAIPM requires only 3098.23 seconds, while MOSEK uses over 20000 seconds. Although the memory usage of MAAIPM is within a constant multiple of that of IBP, the former one is ususaly larger than the latter one. But the right sub-figure in Figure \ref{fig:expNmt} and the case of $N=5000$  demonstrate that MAAIPM's memory usage is managed more efficient compared to Gurobi and MOSEK.
These positive traits are consistent with the time and memory complexity proved in Theorem \ref{cost}.

\begin{wrapfigure}[16]{r}{0.3\linewidth}
\vspace{-10pt}
\includegraphics[width = 1\linewidth]{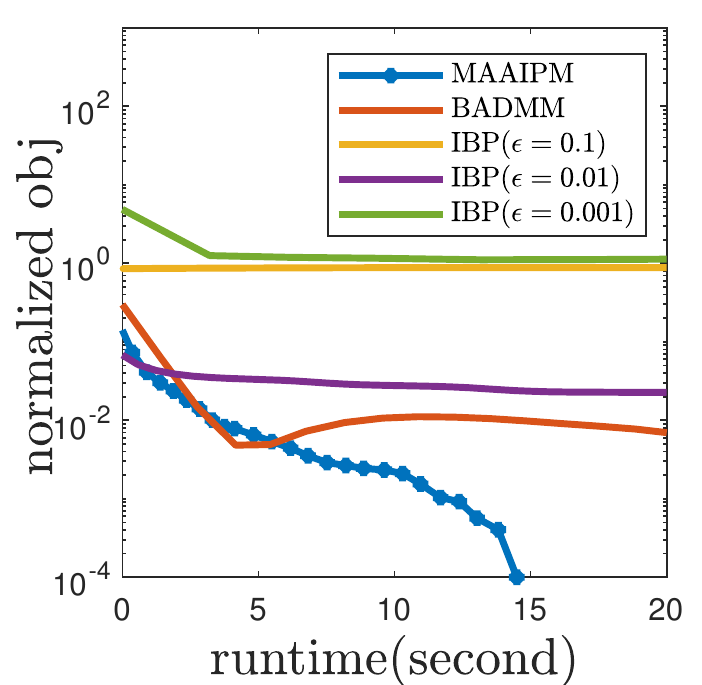}
\caption{Performance of methods in free support cases. $N = 40$, $m=m_1=m_2=\dots=m_N=50$. }
\label{fig:suppadvantage}
\end{wrapfigure}

Next, we conduct numerical studies to test MAAIPM in free support cases, i.e., problem \eqref{non-fix}.
Same as \cite{BADMM}, we implement the version of BADMM and IBP that can automatically update support points and set the initial support points in multivariate normal distribution.
We set the maximum number of iterations in BADMM and IBP as $10^4$ and $10^6$.
The entries of $(\boldsymbol{q}_1^{(t)},\ldots,\boldsymbol{q}_{m'}^{(t)})$ are generated with i.i.d. uniform distribution in $(0,100)$ and the initial support points follows a Gaussian distribution.
In figure \ref{fig:suppadvantage}, "Normalized obj" denotes $\mathcal{F}_{method}/\mathcal{F}_{MAAIPM}-1$, where $\mathcal{F}_{method}$ is the objective value obtained by each iteration of methods.
From figure \ref{fig:unfixsupp} and \ref{fig:suppadvantage}, one can see that, in the free support cases, MAAIPM can still obtain the smallest objective value in the second shortest time.  That is because MAAIPM updates support more frequently and adopts "jump" tricks to avoid the local minima.
Although IBP can obtain an approximate value in the shortest time when $\epsilon = 0.1$, the quality of the barycenter is too low to be useful.


\begin{wraptable}[14]{r}{0.65\linewidth}
        \caption{Experiments on datasets}
        \label{tbl:mnist}
        \centering
        \begin{tabular}{cM{7mm}M{7mm}M{7mm}M{7mm}M{7mm}M{7mm}}
           \toprule
            \multicolumn{1}{c|}{}&\multicolumn{3}{c|}{MNIST}&
            \multicolumn{3}{c}{Fashion-MNIST} \\
            \midrule
            \multicolumn{1}{c|}{time(seconds)} & 250 & 500 & \multicolumn{1}{c|}{1000} & 25  & 50& 75\\
            \midrule
            MAAIPM & \includegraphics[width=1\linewidth]{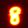} & \includegraphics[width=1\linewidth]{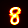} & \includegraphics[width=1\linewidth]{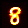} &
            \includegraphics[width=1\linewidth]{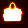}&
            \includegraphics[width=1\linewidth]{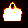}&
            \includegraphics[width=1\linewidth]{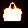}\\
            BADMM & \includegraphics[width=1\linewidth]{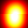} & \includegraphics[width=1\linewidth]{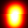} & \includegraphics[width=1\linewidth]{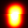} &
            \includegraphics[width=1\linewidth]{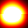}&
            \includegraphics[width=1\linewidth]{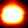}&
            \includegraphics[width=1\linewidth]{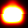}\\
            IBP($\epsilon=0.01$) & \includegraphics[width=1\linewidth]{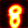} & \includegraphics[width=1\linewidth]{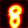} & \includegraphics[width=1\linewidth]{IBP443.png} &
            \includegraphics[width=1\linewidth]{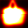}&
            \includegraphics[width=1\linewidth]{IBP23.png}&
            \includegraphics[width=1\linewidth]{IBP23.png}\\
            \bottomrule
        \end{tabular}
    \end{wraptable}

\paragraph{Experiments on real applications:} We conduct similar experiments to  \cite{Cuturi-2014-OTbarycenter,sGS-ADMM} on the MNIST\footnotemark[4] and Fashion-MNIST\footnotemark[4] \footnotetext[4]{Available in \href{http://yann.lecun.com/exdb/mnist/}{http://yann.lecun.com/exdb/mnist/} and \href{https://github.com/zalandoresearch/fashion-mnist}{https://github.com/zalandoresearch/fashion-mnist}} datasets. In MNIST, We randomly select 200 images for digit 8 and resize each image to 0.5, 1, 2 times of its original size $28 \times 28$. In Fashion-MNIST, we randomly select 20 images of handbag, and resize each image to 0.5, 1 time of the original size. {The support points of images are dense and different.} Next, for each case, we apply MAAIPM, BADMM and IBP($\epsilon = 0.01$) to compute the Wasserstein barycenter in respectively free support cases and pre-specified support cases.
From table \ref{tbl:mnist}, one can see that, MAAIPM obtained the clearest and sharpest barycenters within the least computation time. 

\subsubsection*{Acknowledgments} We thank Tianyi Lin, Simai He, Bo Jiang, Qi Deng and Yibo Zeng for helpful discussions and fruitful suggestions.

{
\small
\bibliographystyle{plain}

}

\appendix
\section{Proof of lemma \ref{precondition-lemma}}
To justify the claims in lemma  \ref{precondition-lemma}, we follows the following line of proof: a) First, we show that through a series of row transformations, we can transform matrix A into a matrix whose elements in $(M+1)$-th, $(M+m+1)$-th, $\cdots$, $(M+(N-1)m+1)$-th rows  are zeros, and elements in other positions are the same as A. b) Second, we prove that the matrix $\bar A$ has full row rank.

a). From the definition of matrix A, we have
\begin{equation}\label{matrix A second}
A=
\begin{bmatrix}
F_1                              &                                 &                &                  &   \\
                                 &F_2                              &                &                   &   \\
                                 &                                  &\ddots          &                  &  \\
                                 &                                  &                & F_N &  \\
            G_1 &                     &               &                   & -I_m \\
                                 &G_2                                &                &                  & -I_m \\
                                 &                                   &\ddots           &                &  \vdots\\
                                 &                                   &              & G_N        & -I_m \\
                                 &                                   &                 &                      &  \boldsymbol{1}_m^T \\
\end{bmatrix}
\end{equation}

where $ F_i = I_{m_i}\otimes \boldsymbol{1}_m^{\top} $, $ G_i = \boldsymbol{1}_{m_i}^{\top}\otimes I_m $ for $i=1,...,N$.

Let
\begin{equation*}
  \boldsymbol{e}_1 = \begin{bmatrix} 1 \\ 0 \\ \vdots \\ 0 \end{bmatrix}_{m\times 1}, \ \
  T_i =
  \begin{bmatrix}
  1 & 1 & \cdots & 1 \\ & & & \\ & & & \\ & & & \\
  \end{bmatrix}_{m\times m_i},
   \ \ \
   S_i = \begin{bmatrix}
  1 & 1 & \cdots & 1 \\ & 1 & & \\ & & \ddots & \\ & & & 1 \\
  \end{bmatrix}_{m\times m}, \ i=1,...,N
\end{equation*}
and
\begin{equation*}
  L_1 =
  \begin{bmatrix}
I_{m_1}   &     &      &        &   &  &  \\
         &\ddots   &       &          &  & &  \\
    &      &I_{m_N}   &    & & & \\
  -T_1 &   &    &  I_m     & & & \\
  & \ddots   &     &     & \ddots & & \\
    &        & -T_N  &    & & I_m  &  \\
     &       &      &     & & & 1  \\
\end{bmatrix}, \ \
L_2 =
\begin{bmatrix}
I_{m_1}   &     &      &        &   &  &  \\
         &\ddots   &       &          &  & &  \\
    &      &I_{m_N}   &    & & & \\
   &   &    &  S_1     & & &\boldsymbol{e}_1 \\
  &    &     &     & \ddots & & \vdots \\
    &        &   &    & & S_N  & \boldsymbol{e}_1 \\
     &       &      &     & & & 1  \\
\end{bmatrix}
\end{equation*}

Then
\begin{equation*}
  L_2 L_1 A =
\begin{bmatrix}
F_1                              &                                 &                &                  &   \\
                                 &F_2                              &                &                   &   \\
                                 &                                  &\ddots          &                  &  \\
                                 &                                  &                & F_N &  \\
            G_1^{(1)} &                     &               &                   & H^{(1)} \\
                                 &G_2^{(1)}                                &                &                  & H^{(1)} \\
                                 &                                   &\ddots           &                &  \vdots\\
                                 &                                   &              & G_N^{(1)}        & H^{(1)} \\
                                 &                                   &                 &                      &  \boldsymbol{1}_m^T \\
\end{bmatrix},
\end{equation*}
where $G_i^{(1)} = S_i G_i - S_i T_i F_i $, $ H^{(1)} = \boldsymbol{e}_1 \boldsymbol{1}_m^{\top} -S_i  $. It is easy to verify that elements in the first rows of $H^{(1)}$ and $G_i^{(1)}, i=1,...,N$ are zeros. We have proved the claims in a).

b). As defined in the claims of lemma \ref{precondition-lemma}, $\bar A$ is obtained by removing the $(M+1)$-th, $(M+m+1)$-th, $\cdots$, $(M+(N-1)m+1)$-th rows of A. That is,

$$
\bar A =
  \begin{bmatrix}
F_1                              &                                 &                &                  &   \\
                                 &F_2                              &                &                   &   \\
                                 &                                  &\ddots          &                  &  \\
                                 &                                  &                & F_N &  \\
            G_1^{(2)} &                     &               &                   & H^{(2)} \\
                                 &G_2^{(2)}                                &                &                  & H^{(2)} \\
                                 &                                   &\ddots           &                &  \vdots\\
                                 &                                   &              & G_N^{(2)}        & H^{(2)} \\
                                 &                                   &                 &                      &  \boldsymbol{1}_m^{\top} \\
\end{bmatrix}
$$

where $ G_i^{(2)} = G_i^{(1)}(2:m,:) = \boldsymbol{1}_{m_i}^{\top} \otimes[\boldsymbol{0}_{m-1}, I_{m-1}] $, $H^{(2)} = H^{(1)}(2:m,:) = [\boldsymbol{0}_{m-1}, -I_{m-1}]$ and $ F_i = I_{m_i}\otimes \boldsymbol{1}_m^{\top} $. Let $n_{row}' = M+N(m-1)+1$.

For $i = 1,...,N$, let
\begin{equation*}
  U_i = I_{m_i} \otimes \begin{bmatrix} 1 & -1 & \cdots & -1 \\ & 1 & & \\ & & \ddots & \\ & & & 1 \end{bmatrix}_{m \times m} ,  \
  U_{N+1} = \begin{bmatrix} 1 & -1 & \cdots & -1 \\ & 1 & & \\ & & \ddots & \\ & & & 1 \end{bmatrix}_{m \times m}, \
  R_1 = \begin{bmatrix} U_1 & & & \\ & U_2 & & \\ & & \ddots & \\ & & & U_{N+1} \end{bmatrix}
\end{equation*}
%


then

$$
\bar A R_1 =
  \begin{bmatrix}
F_1^{(3)}                              &                                 &                &                  &   \\
                                 &F_2^{(3)}                              &                &                   &   \\
                                 &                                  &\ddots          &                  &  \\
                                 &                                  &                & F_N^{(3)} &  \\
            G_1^{(3)} &                     &               &                   & H^{(3)} \\
                                 &G_2^{(3)}                                &                &                  & H^{(3)} \\
                                 &                                   &\ddots           &                &  \vdots\\
                                 &                                   &              & G_N^{(3)}        & H^{(3)} \\
                                 &                                   &                 &                      &  \boldsymbol{\alpha}^{\top} \\
\end{bmatrix}
$$
where $ F_i^{(3)} = F_i U_i = I_{m_i}\otimes [1,\boldsymbol{0}_{m-1}^{\top}]$,
$ G_i^{(3)} = G_i^{(2)} U_i = G_i^{(2)} = \boldsymbol{1}_{m_i}^{\top} \otimes[\boldsymbol{0}_{m-1}, I_{m-1}]$, $i=1,...,N$,
$ H^{(3)} = H^{(2)}U_{N+1} = H^{(2)} = [\boldsymbol{0}_{m-1}, -I_{m-1}]$
and $ \boldsymbol{\alpha}^{\top} = \boldsymbol{1}_m^{\top} U_{N+1} = [1,\boldsymbol{0}_{m-1}^{\top}]$.

Let
\begin{equation*}
  \tilde K = \begin{bmatrix} 0 & & & \\ & -1 & & \\ & & \ddots & \\ & & & -1 \end{bmatrix}_{m\times m}, \
   K_i = \begin{bmatrix} I_m & \tilde K & \cdots & \tilde K \\ & I_m & & \\ & & \ddots & \\ & & & I_m \end{bmatrix}_{mm_i\times mm_i}, \
   R_2 = \begin{bmatrix} K_1 & & & \\ & \ddots & & \\ & & K_N & \\ & & & I_m \end{bmatrix}
\end{equation*}

then
$$
\bar A R_1 R_2 =
  \begin{bmatrix}
F_1^{(4)}                              &                                 &                &                  &   \\
                                 &F_2^{(4)}                              &                &                   &   \\
                                 &                                  &\ddots          &                  &  \\
                                 &                                  &                & F_N^{(4)} &  \\
            G_1^{(4)} &                     &               &                   & H^{(3)} \\
                                 &G_2^{(4)}                                &                &                  & H^{(3)} \\
                                 &                                   &\ddots           &                &  \vdots\\
                                 &                                   &              & G_N^{(4)}        & H^{(3)} \\
                                 &                                   &                 &                      &  \boldsymbol{\alpha}^{\top} \\
\end{bmatrix}
$$

where $ F_i^{(4)} = F_i^{(3)}K_i = F_i^{(3)} = I_{m_i}\otimes [1,\boldsymbol{0}_{m-1}^{\top}]$,
$ G_i^{(4)} = G_i^{(3)} K_i =  [\boldsymbol{0}_{m-1}, I_{m-1}, \boldsymbol{0}_{(m-1)\times (mm_i-m)}]$, $i=1,...,N$,

Let $\tilde A $ be the matrix composing of the first $(mM+1)$ columns of $ \bar A R_1 R_2 $. That is,
$$
\tilde A =
  \begin{bmatrix}
F_1^{(4)}                              &                                 &                &                  &   \\
                                 &F_2^{(4)}                              &                &                   &   \\
                                 &                                  &\ddots          &                  &  \\
                                 &                                  &                & F_N^{(4)} &  \\
            G_1^{(4)} &                     &               &                   &  \\
                                 &G_2^{(4)}                                &                &                  &  \\
                                 &                                   &\ddots           &                &  \\
                                 &                                   &              & G_N^{(4)}        &  \\
                                 &                                   &                 &                      &  1 \\
\end{bmatrix}
$$

Matrix $ \tilde A $ satisfies two properties:

(1) Each row of $ \tilde A $ has one and only one nonzero element (being 1) with other elements being 0;

(2) Each column of $ \tilde A $ has at most one nonzero element.

Therefore, there exists permutation matrices $P_1 \in \R^{n_{row}'}$ and $Q_1\in \R^{n_{col}-m+1}$ such that $ P_1 \tilde A Q_1 = [I_{n_{row}'}, 0_{n_{row}' \times (Mm+1) }] $. Thus $ rank(\tilde A) = rank(P_1 \tilde A Q_1) = n_{row}'$ and $rank(\bar A) = n_{row}'$.

\section{Proof of lemma \ref{structure1}}

In this subsection, we give the proof of lemma \ref{structure1}.

\begin{proof}
Let $\boldsymbol{d}$ be the diagonal vector of matrix $ D $; $M:= \sum_{i=1}^N m_i$ and $M_2:= N(m-1)$. Same as the preceding section, the structure of $\bar A$ as:
 \begin{equation*}
   \bar A =
     \begin{bmatrix}
F_1                              &                                 &                &                  &   \\
                                 &F_2                              &                &                   &   \\
                                 &                                  &\ddots          &                  &  \\
                                 &                                  &                & F_N &  \\
            G_1^{(2)} &                     &               &                   & H^{(2)} \\
                                 &G_2^{(2)}                                &                &                  & H^{(2)} \\
                                 &                                   &\ddots           &                &  \vdots\\
                                 &                                   &              & G_N^{(2)}        & H^{(2)} \\
                                 &                                   &                 &                      &  \boldsymbol{1}_m^{\top} \\
\end{bmatrix}
 \end{equation*}

where $ G_i^{(2)} = G_i^{(1)}(2:m,:) = \boldsymbol{1}_{m_i}^{\top} \otimes[\boldsymbol{0}_{m-1}, I_{m-1}] $, $H^{(2)} = H^{(1)}(2:m,:) = [\boldsymbol{0}_{m-1}, -I_{m-1}]$ and $ F_i = I_{m_i}\otimes \boldsymbol{1}_m^{\top} $.

Let
$$
\bar A_1 := \bar A(1:M,:) =
\begin{bmatrix}
F_1                              &                                 &                &                  &   \\
                                 &F_2                              &                &                   &   \\
                                 &                                  &\ddots          &                  &  \\
                                 &                                  &                & F_N &  \\
\end{bmatrix},
$$
$$
\bar A_2 := \bar A(M+1:M+(m-1)N, :) =
\begin{bmatrix}
          G_1^{(2)} &                     &               &                   & H^{(2)} \\
                                 &G_2^{(2)}                                &                &                  & H^{(2)} \\
                                 &                                   &\ddots           &                &  \vdots\\
                                 &                                   &              & G_N^{(2)}        & H^{(2)} \\
\end{bmatrix},
$$

$$
\bar A_3 := \bar A(M+(m-1)N+1,:)  =
\begin{bmatrix}
                 &                                   &                 &                      &  \boldsymbol{1}_m^{\top} \\
\end{bmatrix}.
$$

Then
$$
\bar A=
\begin{bmatrix}
\bar A_1 \\ \bar A_2 \\ \bar A_3
\end{bmatrix} \
and \
\bar A D \bar A^{\top} =
\begin{bmatrix}
\bar A_1 D \bar A_1^{\top} & \bar A_1 D \bar A_2^{\top} & \bar A_1 D \bar A_3^{\top} \\
\bar A_2 D \bar A_1^{\top} & \bar A_2 D \bar A_2^{\top} & \bar A_2 D \bar A_3^{\top} \\
\bar A_3 D \bar A_1^{\top} & \bar A_3 D \bar A_2^{\top} & \bar A_3 D \bar A_3^{\top} \\
\end{bmatrix}.
$$

Now we analyze the structure of each sub-matrix $\bar A_i D \bar A_j^{\top}$ and rename them for conciseness. Let
$$
D =
\begin{bmatrix}
D_1 & & & \\
   & D_2 & & \\
   & & \ddots & \\
   & & & D_{N+1}
\end{bmatrix},
$$

where $D_i\in \R^{mm_i\times mm_i}$, $i=1,\ldots,N$ and $D_{N+1} \in \R^{m\times m}$. Then
$$
\bar A_1 D \bar A_1^{\top} =
\begin{bmatrix}
F_1 D_1 F_1^{\top} & & \\
 & \ddots & \\
 & & F_N D_N F_N^{\top}
\end{bmatrix}
:= B_1.
$$
Each $F_i D_i F_i^{\top}$ is a diagonal matrix with positive diagonal entries.

$$
\bar A_2 D \bar A_1^{\top} =
\begin{bmatrix}
G_1^{(2)}D_1 F_1^{\top} & & \\
 & \ddots & \\
 & & G_N^{(2)}D_N F_N^{\top}
\end{bmatrix}
:= B_2^{\top},
$$

\begin{equation}\label{A22}
\bar A_2 D \bar A_2^{\top} =
\begin{bmatrix}
G_1^{(2)} D_1 G_1^{(2)\top} & & \\
 & \ddots & \\
 & & G_N^{(2)} D_N G_N^{(2)\top}
\end{bmatrix}
+
\begin{bmatrix}
H^{(2)} D_{N+1} H^{(2)\top} & \cdots & H^{(2)} D_{N+1} H^{(2)\top} \\
\vdots & & \vdots \\
H^{(2)} D_{N+1} H^{(2)\top} & \cdots & H^{(2)} D_{N+1} H^{(2)\top}
\end{bmatrix}.
\end{equation}

where $H^{(2)} D_{N+1} H^{(2)\top}$ and each $G_i^{(2)} D_i G_i^{(2)\top}$ is a diagonal matrix with positive diagonal entries. We use $B_3 $ to denote the first matrix in the right hand side of (\ref{A22}) and $B_4$ to denote the second. In addition, other blocks of $\bar A D \bar A^{\top}$ are

$$
\bar A_3 D \bar A_1^{\top} = 0,
$$

$$
\bar A_3 D \bar A_2^{\top} =
\begin{bmatrix}
\boldsymbol{1}_m^{\top} D_{N+1} H^{(2)\top} & \cdots & \boldsymbol{1}_m^{\top} D_{N+1} H^{(2)\top}
\end{bmatrix}
:= \boldsymbol{\alpha}^{\top},
$$

$$
\bar A_3 D \bar A_3^{\top} = \boldsymbol{1}_m^{\top} D_{N+1} \boldsymbol{1}_m := c.
$$

With the new notations, we have
$$
\bar A D \bar A^T =
\begin{bmatrix}
B_1  &  B_2  &  \boldsymbol{0} \\
B_2^{\top} & B_3 + B_4 & \boldsymbol{\alpha} \\
\boldsymbol{0} & \boldsymbol{\alpha}^{\top} & c
\end{bmatrix}.
$$

\end{proof}

\section{Proof of lemma \ref{basic2}}
To justify lemma \ref{basic2}, we need the following basic result which can be verified through direct computation.

\begin{lemma}\label{basic1}
 All the non-zero entries of matrices $B_1,B_2,B_3$ and $B_4$ are positive, and

\textbf{a)}  $B_3 \boldsymbol{1}_{M_2}= B_2^{\top} \boldsymbol{1}_M$. \quad
\textbf{b)}  $ B_1\boldsymbol{1}_M - B_2 \boldsymbol{1}_{M_2} >0 $.
\end{lemma}
\begin{proof}
a)
$$
B_3 \boldsymbol{1}_{M_2} =
\begin{bmatrix}
G_1^{(2)} D_1 G_1^{(2)\top} \boldsymbol{1}_{m-1} \\
\vdots \\
G_N^{(2)} D_N G_N^{(2)\top} \boldsymbol{1}_{m-1}
\end{bmatrix}, \ \
B_2^{\top} \boldsymbol{1}_M =
\begin{bmatrix}
G_1^{(2)}D_1 F_1^{\top} \boldsymbol{1}_{m_1} \\
\vdots \\
G_N^{(2)}D_N F_N^{\top} \boldsymbol{1}_{m_N}
\end{bmatrix}
$$
Recall that $ G_i^{(2)}  = \boldsymbol{1}_{m_i}^{\top} \otimes[\boldsymbol{0}_{m-1}, I_{m-1}] $, $F_i = I_{m_i}\otimes \boldsymbol{1}_m^{\top}$, and $D_i$'s are diagonal matrices, we have $G_i^{(2)}D_i F_i^{\top} \boldsymbol{1}_{m_i} =  G_i^{(2)} D_i G_i^{(2)\top} \boldsymbol{1}_{m-1}$ and thus $B_3 \boldsymbol{1}_{M_2}= B_2^{\top} \boldsymbol{1}_M$.

b)
$$
B_1\boldsymbol{1}_M =
\begin{bmatrix}
F_1 D_1 F_1^{\top} \boldsymbol{1}_{m_1} \\
\vdots \\
F_N D_N F_N^{\top} \boldsymbol{1}_{m_N}
\end{bmatrix}, \ \
B_2 \boldsymbol{1}_{M_2} =
\begin{bmatrix}
F_1 D_1 G_1^{(2)\top} \boldsymbol{1}_{m-1} \\
\vdots \\
F_N D_N G_N^{(2)\top} \boldsymbol{1}_{m-1}
\end{bmatrix}
$$
It is easy to verify that $F_i D_i F_i^{\top} \boldsymbol{1}_{m_i} > F_i D_i G_i^{(2)\top} \boldsymbol{1}_{m-1}$ and thus $ B_1\boldsymbol{1}_M - B_2 \boldsymbol{1}_{M_2} >0 $.
\end{proof}
With this basic lemma at hand, we are able to prove lemma \ref{basic2}.

\textbf{proof of lemma \ref{basic2}: }
\begin{proof}

\textbf{a)} It is easy to verify the block-diagonal structure of $A_1$, so we just need to prove the positive definiteness and the strict diagonal dominance.
Assume $A_1$ is not positive definite and $-\lambda\le 0 $ is an eigenvalue of $A_1$, then $ \lambda I_{M_2}+A_1 $ is a singular matrix.

From the results in lemma \ref{basic1}, we have
\begin{eqnarray}\label{vec ieq}
& & (\lambda I_{M_2} +A_1)\boldsymbol{1}_{M_2}\nonumber\\
&=&\lambda\boldsymbol{1}_{M_2}+ B_3\boldsymbol{1}_{M_2} - (B_2^{\top} B_1^{-1} B_2)\boldsymbol{1}_{M_2}\nonumber\\
&=& \lambda\boldsymbol{1}_{M_2}+B_2^{\top}B_1^{-1}B_1\boldsymbol{1}_{M} - (B_2^{\top} B_1^{-1} B_2)\boldsymbol{1}_{M_2}\nonumber\\
&=& \lambda\boldsymbol{1}_{M_2}+B_2^{\top}B_1^{-1}( B_1\boldsymbol{1}_{M} - B_2\boldsymbol{1}_{M_2} )\nonumber\\
&>& \boldsymbol{0}_{M_2}
\end{eqnarray}

where the first equality is from a) of lemma \ref{basic1}; the last inequality is from b) of lemma \ref{basic1} and the fact that $B_2^{\top}B_1^{-1}\ge 0$ and each row of $B_2^{\top}B_1^{-1}$ has at least one strict positive entry.

Since $B_1,B_2,B_3 \ge0$, $B_3$ is a diagonal matrix, together with (\ref{vec ieq}), we know that the diagonal entries of $\lambda I_{M_2}+A_1 = \lambda I_{M_2}+ B_3 - B_2^{\top} B_1^{-1} B_2$ are positive and the off-diagonal entries are non-positive.
Let $E_{M_2}:= \boldsymbol{1}_{M_2}\boldsymbol{1}_{M_2}^{\top} - I_{M_2}$, then
$$
I_{M_2}\circ \left| A_1+\lambda I_{M_2} \right| = I_{M_2}\circ \left( A_1+\lambda I_{M_2} \right) , \ \
E_{M_2}\circ \left| A_1+\lambda I_{M_2} \right| = - E_{M_2}\circ\left( A_1+\lambda I_{M_2} \right),
$$
and
$$
\left( I_{M_2}\circ \left| A_1+\lambda I_{M_2} \right| \right)\boldsymbol{1}_{M_2} - \left( E_{M_2}\circ \left| A_1+\lambda I_{M_2} \right| \right)\boldsymbol{1}_{M_2} = (\lambda I_{M_2} +A_1)\boldsymbol{1}_{M_2} > \boldsymbol{0}_{M_2}
$$
This means $\lambda I_{M_2}+A_1$ is strictly diagonal dominant and thus nonsingular, which is a contradiction. Therefore, $A_1$ is positive definite. Take $\lambda=0$ in the preceding analysis, we know $A_1$ is strictly diagonal dominant.

\textbf{b)} It is easy to verify that $ A_2 = (\boldsymbol{1}_N\boldsymbol{1}_N^{\top})\otimes (diag(\boldsymbol{y})- \frac{1}{c}\boldsymbol{y}\boldsymbol{y}^{\top} ) $.
In view of the definition of $c$, we have $c >\boldsymbol{1}_{m-1}^{\top}\boldsymbol{y} $. Thus, the second claim of b) is a special case of a) with $ B_1 = c $, $ B_2 = \boldsymbol{y}^{\top}$ and $ B_3 = diag(\boldsymbol{y})$.

\end{proof}

\section{Analysis of algorithm \ref{Low-rank regularization method}}
In this section, we prove that through the steps in Algorithm \ref{Low-rank regularization method}, we get the accurate solution of the system $ (A_1+A_2)\boldsymbol{x} = \boldsymbol{g} $. We need a basic lemma on the inverse matrix on the sum of tow matrices.

\begin{lemma}\label{low rank inv}
Let $A\in \R^{n\times n}$ be an nonsingular matrix and $B\in \R^{n\times d}$, where $n$ and $d$ are two positive integers. Then
$$
(A+BB^{\top})^{-1} = A^{-1} - A^{-1}B (I_n + B^{\top}A^{-1}B)^{-1}B^TA^{-1}
$$
\end{lemma}

Recall that we have proved in lemma \ref{basic2} that $Y$ is positive definite. Suppose $ Y = R^{\top}R ,R\in \R^{(m-1)\times(m-1)}$ and let $ \tilde R = \boldsymbol{1}_N\otimes R^{\top} $. Then,$ A_2 = \tilde R \tilde R^{\top}$. Further more, let
$$
\bar{R}=
\begin{bmatrix}
0 \\  \vdots \\ 0 \\ 1
\end{bmatrix}_{N\times 1}
\otimes R^{\top},\ and \ \
U =
\begin{bmatrix}
I_{m-1} & & & I_{m-1}\\
       & I_{m-1} & & \vdots \\
       &        & \ddots & I_{m-1} \\
       & & & I_{m-1}
\end{bmatrix}_{M_2\times M_2}
$$
Note that $U$ is the same as defined in the main text part above Algorithm \ref{Low-rank regularization method}. It is easy to verify that $\tilde{R} = U\bar{R}$ and with the help of lemma \ref{low rank inv}, we have
\begin{eqnarray*}
  \left(A_1+A_2 \right)^{-1}
  &=& \left( A_1 + \tilde{R} \tilde{R}^{\top} \right)^{-1}  \\
   &=& A_1^{-1} - A_1^{-1}\tilde{R}(I + \tilde{R}^{\top} A_1^{-1} \tilde{R})^{-1} \tilde{R}^{\top} A_1^{-1} \\
   &=& A_1^{-1} - A_1^{-1}U\bar{R}(I + \tilde{R}^{\top} A_1^{-1} \tilde{R})^{-1} \bar{R}^{\top} U^{\top}A_1^{-1}. \\
\end{eqnarray*}

Define

\begin{equation}\label{W}
W := \bar{R}(I + \tilde{R}^{\top} A_1^{-1} \tilde{R})^{-1} \bar{R}^{\top} =
\begin{bmatrix} 0 \\  \vdots \\ 0 \\ R^{\top} \end{bmatrix}
(I_{m-1}+ \sum_{i=1}^N  R A_{ii} R^{\top})^{-1}
\begin{bmatrix} 0 & \cdots & 0 & R \end{bmatrix}
\end{equation}

 then
 \begin{equation}\label{inv}
\left(A_1 + A_2 \right)^{-1} = A_1^{-1} - A_1^{-1}U WU^{\top}A_1^{-1}
\end{equation}

From (\ref{W}), it is clear that all entries of $W$ are zero, except for the last $(m-1)\times (m-1)$ block $W_{NN} = R^{\top} (I_{m-1}+ \sum_{i=1}^N R A_{ii} R^{\top})^{-1} R$. With further calculation,
$$
W_{NN} = \big(R^{-1}R^{-\top} + \sum_{i=1}^{N}A_{ii}^{-1} \big)^{-1} = \big(Y^{-1} + \sum_{i=1}^{N}A_{ii}^{-1} \big)^{-1}.
$$

To solve the system $ (A_1+A_2)\boldsymbol{x} = \boldsymbol{g} $ with the equation (\ref{inv}), we just need to let each term in (\ref{inv}) act on the vector step by step. That's exactly what Algorithm \ref{Low-rank regularization method} does.

\section{Proof of theorem \ref{cost}}
In this section, we present the detailed analysis of computational cost and memory usage of SLRM and DLRM. Here we restate theorem \ref{cost}.
\begin{theorem}\label{cost-appendix}
1). For SLRM, Algorithm \ref{Normal equation solver}, the time complexity in terms of flops is $O(m^2 \sum_{i=1}^N m_i + Nm^3)$, and the memory usage in terms of doubles is $O(m\sum_{i=1}^N m_i + Nm^2)$.
2). For the DLRM, the time complexity in terms of flops is $O( m \sum_{i=1}^N m_i^2 + \sum_{i=1}^N m_i^3)$, and the memory usage in terms of doubles is $O(m\sum_{i=1}^N m_i + \sum_{i=1}^N m_i^2)$.
\end{theorem}

\begin{proof}
(1) First, for SLRM, assuming taking full advantage of the sparse structure, we count the flops required for computing each of the following quantities in Algorithm \ref{Normal equation solver}:
\begin{equation*}
    B_1: O(m \sum_{t=1}^N m_t); \
    B_2: 0; \
    B_3: O(m \sum_{t=1}^N m_t); \
    T: O(m \sum_{t=1}^N m_t); \
    A_1: O(m^2 \sum_{t=1}^N m_t); \
    A_2: O(m^2);
\end{equation*}
\begin{equation*}
    \boldsymbol{z}^{(1)}: O(m \sum_{t=1}^N m_t); \
    \boldsymbol{z}^{(2)}: O(Nm); \
    \boldsymbol{z}^{(3)}: O( Nm^3); \
    \boldsymbol{z}^{(4)}: O(Nm); \
    \boldsymbol{z}^{(5)}: O(m \sum_{t=1}^N m_t).
\end{equation*}
The computation of $A_1$ and $\boldsymbol{z}^{(3)}$ requires most flops. The total flops required for SLRM is $O(m^2 \sum_{t=1}^N m_t + Nm^3)$.

On the other hand, for implementation of the whole interior-point methods, the major data that should be kept in the memory include:

(a) Several vectors that is at the same level as a primal variable or a dual variable. Note that the scale of a primal variable is $ m(\sum_{i=1}^N m_i) + m  $ flops, and the scale of a dual variable is  $ \sum_{i=1}^N m_i + N(m-1) +1 $ flops.

(b) Matrix $\bar A$ which is defined in lemma \ref{precondition-lemma}. Recall that
$$
\bar A =
  \begin{bmatrix}
F_1                              &                                 &                &                  &   \\
                                 &F_2                              &                &                   &   \\
                                 &                                  &\ddots          &                  &  \\
                                 &                                  &                & F_N &  \\
            G_1^{(2)} &                     &               &                   & H^{(2)} \\
                                 &G_2^{(2)}                                &                &                  & H^{(2)} \\
                                 &                                   &\ddots           &                &  \vdots\\
                                 &                                   &              & G_N^{(2)}        & H^{(2)} \\
                                 &                                   &                 &                      &  \boldsymbol{1}_m^{\top} \\
\end{bmatrix}
$$
Since each column of $F_i$ and each column of $G^{(i)}$ has at most one non-zero element, the total number of non-zero elements in $F_1,\dots, F_N$ and $ G_1^{(2)} ,\dots, G_N^{(2)} $ is bounded by $2m \sum_{i=1}^N m_i$. In addition, $H^{(2)}$ has $m-1$ non-zero elements, so the total number of non-zero elements in $\bar A$ is bounded by $2m \sum_{i=1}^N m_i + N(m-1) + m$

(c) Diagonals of matrices $B_1$ and $B_3$, and diagonal blocks of matrices $B_2$ and $A_1$. The data scale of the diagonals of matrices $B_1$ and $B_3$ are even smaller than a dual variable. The diagonal blocks of matrices $B_2$ and $ A_1$ have $m\sum_{i=1}^N m_i$ elements and $N(m-1)^2$ elements, respectively.

(d) Other intermediate vectors or matrices, whose data scale is bounded by a constant time of the data scale in (a), (b) and (c).

With the analysis in (a) to (d), we know the memory usage of SLRM is bounded by $O(m\sum_{i=1}^N m_i + Nm^2)$.

(2) The major difference of DLRM and SLRM is that, we don't need to formulate the diagonal blocks $\{A_{ii}: i=1\dots,N\}$ of matrix $A_1$ explicitly and compute the inverses of $A_{ii}s$. Instead, we need to compute $(B_{1i} - B_{2i}B_{3i}^{-1}B_{2i}^{\top})^{-1}$ explicitly, which requires $O(m \sum_{i=1}^N m_i^2)$ flops for matrix multiplication and $O( \sum_{i=1}^N m_i^3)$ flops for matrix inverse. Since all other matrix-vector operations are cheap compared with matrix multiplication and inverse, as a result, the leading cost of the computation time is at the level $O(m \sum_{i=1}^N m_i^2 + \sum_{i=1}^N m_i^3)$.

Further more, since we need to keep $(B_{1i} - B_{2i}B_{3i}^{-1}B_{2i}^{\top})^{-1}, i=1,\dots,N$ in memory instead of $A_{ii}^{-1}$, with simply different analysis as in part(1), we know the memory usage of DLRM is at the level $O(m\sum_{i=1}^N m_i + \sum_{i=1}^N m_i^2)$.

\end{proof}

\section{Examples of local minima and saddle points in free support cases}

\textbf{An example of local minima}

Set $\Pi^{(t)} = \big[{\boldsymbol{\pi}^{(t)}_1}^{\top},{\boldsymbol{\pi}_2^{(t)}}^{\top},\ldots,{\boldsymbol{\pi}_m^{(t)}}^{\top}\big]^{\top}$.  Let $N$ be any positive integer and $m=2$, $d=1$, $m_t = 3$, $Q^{(t)} = [0,0.9,1.1]$ and $\boldsymbol{a}^t = [0.01,0.495,0.495]$. Then $X = [0,1]$, $\boldsymbol{w} = (0.01,0.99)$ and $\boldsymbol{\pi}^{(t)}_1 = (  0.01 , 0 , 0 ) $ and $\boldsymbol{\pi}^{(t)}_2  = ( 0 , 0.495 , 0.495)$ is a local minimum.  But it is not a global minimum because  a lower objective value occurs  when $X = \{0.9,1.1\}$, $\boldsymbol{w} = (0.505,0.495)$, $\boldsymbol{\pi}^{(t)}_1 = (  0.01 , 0.495 , 0 ) $ and $\boldsymbol{\pi}^{(t)}_2  = ( 0 , 0 , 0.495)$.

\textbf{An example of saddle point}

Let $N$ be any positive integer and  $m=2$, $d=1$, $m_t = 3$, $Q^{(t)} = [0,1/2,3/2]$ and $\boldsymbol{a}^t = [1/3,1/3,1/3]$, then $X = [0,1]$, $\boldsymbol{w} = (1/3,2/3)$, $\boldsymbol{\pi}^{(t)}_1 = (  1/3 , 0 , 0 ) $ and $\boldsymbol{\pi}^{(t)}_2  = ( 0 , 1/3 , 1/3)$ is a saddle point. Fixing $X$, the $\boldsymbol{w}$ and $\Pi^{(t)}$ is an optimal basic solution of problem \ref{fix}. Fixing $\boldsymbol{w}$ and $\Pi^{(t)}$, $X$ is the solution of \eqref{optimalX}. It is not a local minimum, because  a lower objective value of problem \ref{fix} can occur when $X = \{\delta,1\}$, $\forall ~ \delta \in (0,1/2)$.

\section{Details of MAAIPM}

Figure \ref{fig:MAAIPM} visualizes the primal variables $x_i$ and objective gradients $c_i$ in each iteration of MAAIPM.

\begin{figure}[htbp]
    \centering
    \includegraphics[width=0.7\textwidth]{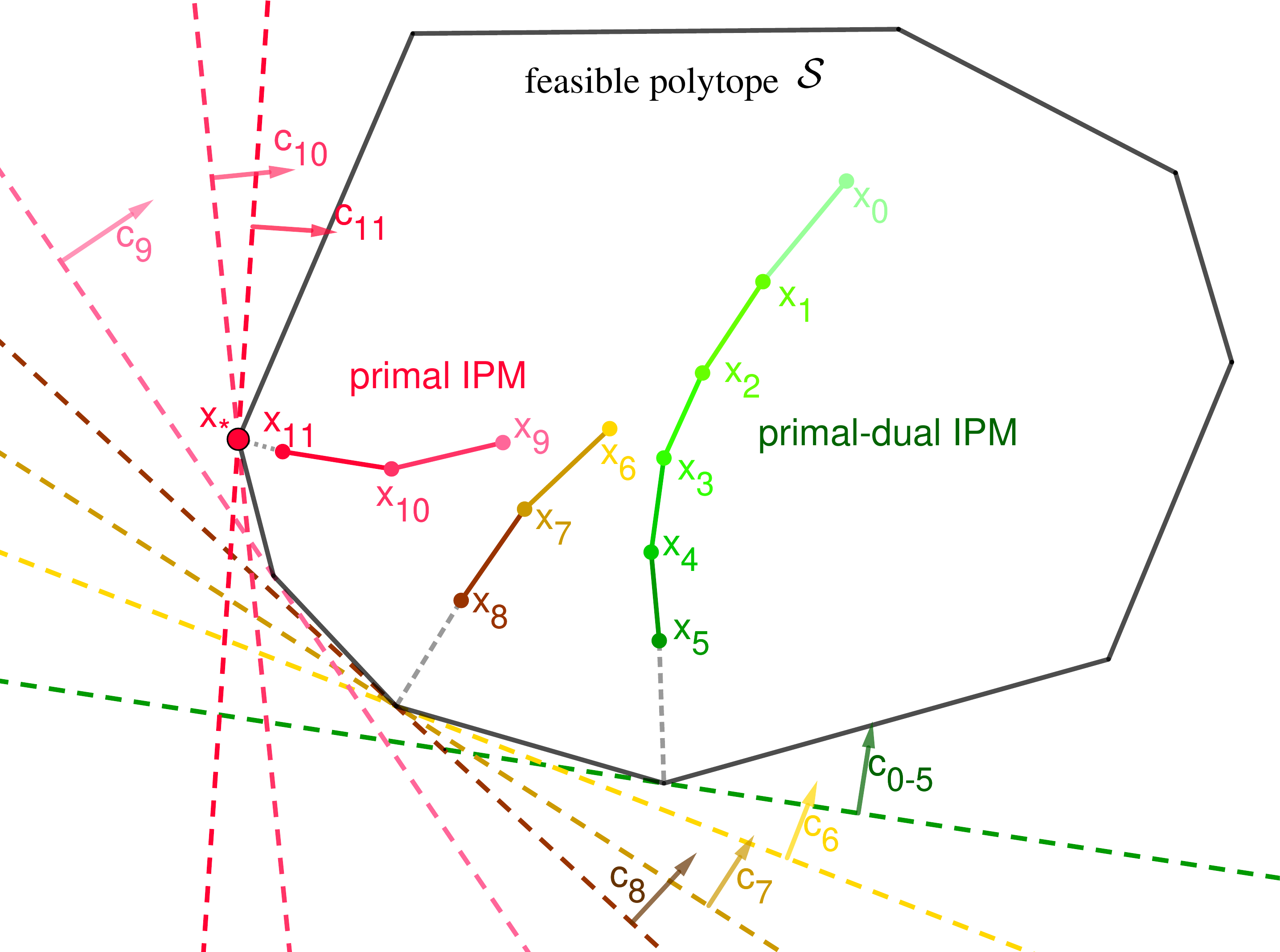}

    \caption{The primal variables and objective gradients in different iterations of MAAIPM. $x_i$ is returned by each iteration of IPM under the objective gradient $c_i$, and $c_i, ~i = 5,\dots,11,$ is calculated by $x_{i-1}$ according to \eqref{re-nonfix}. At the beginning, MAAIPM updates objective gradient after every a few primal-dual IPM iterations(green). Then MAAIPM applies primal IPM(yellow and red) to frequently update objective gradient $c$ and uses "jump" tricks to escape local minima. $x_6$ and $x_9$ are the first primal variables returned by one primal IPM iteration form a smartly chosen starting point.}
    \label{fig:MAAIPM}
\end{figure}

\begin{algorithm}[H]
	\caption{Matrix-based Adaptive Alternating Interior-point Method(MAAIPM)}

    \KwIn{an initial $X^0$}

     \If{support points are pre-specfied}{
        implement predictor-corrector IPM\;
         \textbf{Output} $\boldsymbol{w}^*,\{\Pi^{(t),*}\}$
    }\Comment{Pre-specified support cases}

    \While{at the beginning}{
          predictor-corrector IPM to solve \eqref{fix} and update $X^*$\;
    }\Comment{Update support $X^*$ every a few IPM iterations}

     \While{ a termination criterion is not met}{

     s = 0, apply the warm-start strategy to smartly choose the starting point\;

     \While{ the penalty $\mu^s$ is not sufficiently close to $0$}{

     calculate the Newton direction $p^s$ at $(\boldsymbol{w}^s,\{\Pi^{(t),s}\})$ by \eqref{Newton direction};

     $(\boldsymbol{w}^{s+1},\{\Pi^{(t),s+1}\}) = (\boldsymbol{w}^s,\{\Pi^{(t),s}\}) + \alpha^sp^s$, where $\alpha^s$ ensures the interior point\;

     update $X^*$ by (\ref{optimalX}) and choose penalty $\mu^{s+1} < \mu^s$;

     \Comment{Update support $X^*$ every IPM iteration}

     $s= s+1$;

     }

     }\Comment{"Jump" tricks}

     \KwOut{ $\boldsymbol{w}^s,X^*,\{\Pi^{(t),s}\}$}
\end{algorithm}

\end{document}